\documentclass[12pt,reqno]{amsart}
\usepackage{enumerate, amsmath, amsthm, amsfonts, amssymb, xy,  mathrsfs, graphicx, paralist, lscape, array, mathptmx, mathtools}
\usepackage[usenames, dvipsnames]{color}
\usepackage{hyperref}
\usepackage{float}
\usepackage{linegoal}
\usepackage{graphicx}
\usepackage{inputenc}
\usepackage{subcaption}
\usepackage{filecontents}
\usepackage{graphics}
\usepackage{xcolor}
\usepackage{relsize}
\usepackage[margin=1in]{geometry}
\usepackage{tikz-cd}
\usepackage{fancyhdr,color}
\usepackage{tikz,tikz-cd,graphicx,multicol}
\usetikzlibrary{shapes.geometric}
\usepackage{euscript,eufrak,nicefrac,enumitem}
\usepackage{amsfonts,amsmath,amsthm}
\usepackage{pictexwd,dcpic}
\allowdisplaybreaks
\usepackage{comment}
\usetikzlibrary{knots,matrix,arrows,decorations.markings,decorations.pathreplacing}

\theoremstyle{plain}

\numberwithin{equation}{section}
\newtheorem{theorem}{Theorem}[section]
\newtheorem{proposition}{Proposition}[theorem]
\newtheorem{lemma}[theorem]{Lemma}
\newtheorem{corollary}[equation]{Corollary}
\newtheorem{conjecture}[equation]{Conjecture}

\theoremstyle{definition}
\newtheorem{definition}{Definition}[section]

\theoremstyle{definition}
\newtheorem{remark}{Remark}[section]

\DeclareMathOperator{\Hom}{Hom}
\DeclareMathOperator{\End}{End}

\newcommand{\CC}{\mathbb{C}}

\newcommand{\C}{\mathcal{C}}

\DeclareMathOperator{\ev}{ev}
\DeclareMathOperator{\coev}{coev}
\DeclareMathOperator{\id}{id}

\newcommand{\be}{\begin{equation}}
\newcommand{\ee}{\end{equation}}

\newcommand{\KP}[1]{
  \begin{tikzpicture}[baseline=-\dimexpr\fontdimen22\textfont2\relax]
  #1
  \end{tikzpicture}
}

\newcommand{\OPxa}{
  \KP{
    \draw[color=black,thick,->] (-0.4,-0.4) -- (0.4,0.4);
    \draw[color=black,thick, <-] (-0.4,0.4) -- (-0.1,0.1);
    \draw[color=black,thick] (0.1,-0.1) -- (0.4,-0.4);
  }
}

\newcommand{\OIdma}{
  \KP{
    \draw[color=black,thick,->] (-0.4,-0.4) .. controls (0.01,0) .. (-0.4,0.4);
    \draw[color=black,thick,->] (0.4,-0.4) .. controls (-0.01,0) .. (0.4,0.4);
  }
}
 
\title{A comparison between $SL_n$ spider categories}
\author{Anup Poudel}
\address{Department of Mathematics, The Ohio State University}
\email{poudel.33@osu.edu}

\subjclass{57K31, 17B37}
\keywords{\textit{Webs, HOMFLYPT relations, skein categories, skein modules}}

\begin{document}
\maketitle

\begin{abstract}
    We prove a conjecture of L\^{e} and Sikora by providing a comparison between various existing $SL_n$ skein theories. While doing so, we show that the full subcategory of the spider category, $\mathcal{S}p(SL_n)$, defined by Cautis-Kamnitzer-Morrison, whose objects are monoidally generated by the standard representation and its dual, is equivalent as a spherical braided category to Sikora's quotient category. This also answers a question from Morrison's Ph.D. thesis. Finally, we show that the skein modules associated to the CKM and Sikora's webs are isomorphic.
    
    \end{abstract}

\section{Introduction}
The category of representations of the quantum group $U_q(\mathfrak{sl_n})$ has a spherical and braided tensor (ribbon) structure. In particular, since it is a pivotal monoidal category one can describe the category using diagrammatic calculus. By introducing the notion of \textit{combinatorial spiders} in \cite{kup1}, Kuperberg    
first provided a diagrammatic presentation  for the category of finite dimensional representations of $U_q(\mathfrak{g})$, where $\mathfrak{g}$ is a simple Lie algebra of rank $2$. The diagrammatic presentation  for a representation category has many advantages. For example, diagrammatic presentations lead naturally to the definition of  skein modules. \textit{Skein modules} (c.f. Def. \ref{skeinmoddef}) have become central objects of study in the field of quantum topology connecting them to quantum invariants of $3$-manifolds, topological quantum field theory, quantum cluster algebras and quantum hyperbolic geometry, see for example \cite{bw1, bw2, bfk, fg, fkl, le, mu, ps} and references within. Using diagrammatic presentation for a representation category of a quantum group, one obtains a natural description for its associated skein category (c.f. Section \ref{skeinmodiso}) which allows one to understand the associated skein modules.

Extending Kuperberg's work, Kim  \cite{kim} proposed a presentation of the category of finite dimensional representations of   $U_q(\mathfrak{sl_4})$ where the colors correspond to the exterior powers of the standard representation and its dual. Sikora in \cite{sik} provided a presentation for the braided spherical monoidal category coming from the representation theory of $U_q(\mathfrak{sl}_n)$ using the standard representation and its dual as objects. Further, Morrison proposed a complete set of generators and relations (conjecturally) in \cite{mor} for the spherical monoidal category, $\mathcal{R}ep(U_q(\mathfrak{sl}_n))$  where the colors correspond to the exterior powers of the standard representation and its dual.  Later, Cautis, Kamnitzer and Morrison proved Morrison's conjecture in \cite{ckm} using  skew-Howe duality.

The braided monoidal structure on $U_q(\mathfrak{sl}_n)$ was first explored diagrammatically by Murakami et al. in \cite{moy} (also see \cite{kw}).  They provide web relations that align with the untagged relations (\ref{ckmb1}--\ref{boxrelationsckm}) in \cite{ckm}. However, they provide no discussion of a complete set of relations for this category. Later, Sikora \cite{sik} explained the connection between his presentation for $\mathcal{R}ep(U_q(\mathfrak{sl}_n))$ and generators and relations presented in \cite{moy}. Further, in his thesis \cite{mor}, Morrison poses a question regarding the relation between his conjecture and the work of Sikora. In this paper, we answer Morrison's question and also prove Conjecture \ref{lesikconjecture} \cite{lesik} which is related to the question posed by Morrison in his thesis.

There is a braided spherical category based on the HOMFLYPT skein relations.  Early on it was realized \cite{TW} that by specializing the variables in HOMFLYPT one could obtain a category that mapped down to the categories of $U_q(\mathfrak{sl}_n)$ representations. One can build skeins that behave algebraically like Young symmetrizers \cite{Y, M,  MA, Li, bla}.  The category is missing both generators and  relations that say that the $n^{\text{th}}$ exterior power of the standard representation and its dual are trivial.  Sikora's model adds  $n$-valent vertices that are sources and sinks corresponding to these invariant tensors and a relation for cancelling them.  The CKM model adds tags that are sources and sinks and relations for moving them and cancelling them. The work in this paper shows that the two approaches are equivalent.


As in \cite{lesik}, let $\mathfrak{S}_n^b$ be a monoidal category with finite sequences of signs $\pm$ as objects and isotopy classes of $n-$tangles (cf. \cite{sik}) as morphisms. The tensor product is given by horizontal concatenation and composition of morphisms is given by vertical stacking. The category of modules in \cite{lesik} are over a commutative ring with a distinguished invertible element.

Let $\mathcal{C}_n$ be the category of left $U_q(\mathfrak{sl}_n)$-modules isomorphic to finite tensor products of $V$ and $V^*$ where $V$ is the defining representation of $U_q(\mathfrak{sl}_n)$. Define a monoidal functor $RT_0: \mathfrak{S}_n^b \to \mathcal{C}_n$ which for any object $\eta = \{\eta_1, \cdots \eta_k\}\in \mathfrak{S}_n^b$, is defined as $RT_0(\eta) := V^\eta = V^{\eta_1} \otimes \cdots \otimes V^{\eta_k}$. Note that $V^+ = V$ and $V^- = V^*$. For any $n-$tangle, the functor takes caps and cups to evaluation and coevaluation maps respectively, crossings to the braid isomorphisms and an $n-$sink (resp. source) to a map from the $n-$fold tensor of $V$ (resp. ground ring) to the ground ring (resp. $n-$fold tensor of $V$). Also, a \textit{monoidal ideal} in a monoidal category, $\mathcal{C}$ is a subset $I\subset \Hom(\mathcal{C})$ such that for $x\in I$ and $y\in \Hom(\mathcal{C})$, we have $x\otimes y, y \otimes x \in I$ and $x\circ y, y\circ x \in I$ whenever such compositions are defined. 

\begin{conjecture}[\cite{lesik}] \label{lesikconjecture}
The kernel $\ker RT_0$ is the monoidal ideal generated by elements given in relations (\ref{homflysik}--\ref{unknotsik}).
\end{conjecture}

In this paper, we prove the Conjecture \ref{lesikconjecture},  over an integral domain $R$ where certain quantized integers are invertible (c.f. Section \ref{prelim}), by proving that Sikora's braided spherical category is equivalent (as a braided spherical category) to the full subcategory of the braided spherical category $\mathcal{S}p(SL_n)$ in \cite{ckm} which has as objects the standard representation and its dual.

\subsection{Main results:}
The main results of this paper are:
\begin{itemize}
    \item Theorem \ref{sikorackmequiv} which shows that the braided spherical category coming from \cite{sik} and the full subcategory of the spider category in \cite{ckm} with the standard representation and its dual as objects are equivalent to each other.
    \item Theorem \ref{proofofconjecture} which provides a proof for the Conjecture \ref{lesikconjecture} under our choice of integral domain. 
    \item Theorem \ref{bigckmskeinmodequiv} which shows that the skein modules associated to the Sikora webs is isomorphic to those associated to the CKM webs.
\end{itemize}

\subsection{Outline}
In Section \ref{prelim}, we define the quantized integers (and binomial coefficients) along with the categorical structures that appear in our work. The notion of a free spider category and operations in this category are also introduced. We recall the definitions of the two main categories in this paper: the CKM spider category and Sikora's spherical braided category in Section \ref{ckmsikdef}. We also recall the definition of the MOY category which is closely related to the CKM category. 
In Section \ref{boxes}, the CKM box relations are derived in a diagrammatic fashion using the braided structure of the CKM spider category showing that some of the CKM relations in the braided setting are redundant. This allows us to relate the CKM and Sikora categories easily later in the paper.
In Section \ref{fullsubcat}, we introduce and work with subcategories of (the full subcategory) the MOY category to show that relations in the MOY category is completely characterized by the specialized HOMFLYPT relation (c.f. Figure \ref{blanchethomfly}) in Theorem \ref{reducecolorthm}. This serves as the first step toward relating the CKM and Sikora categories.
In Section \ref{siktockmmain}, with the aid of the main theorem from the previous section and the antisymmetrizer relation (c.f. \ref{antisym}) we prove that Sikora's spherical braided category is equivalent (as a spherical braided category) to the full subcategory of the CKM  spider category. Further, using this result, we prove the Conjecture \ref{lesikconjecture} of L\^{e} and Sikora.
In Section \ref{skeinmodiso}, we extend our result regarding the ribbon categories to an equivalence of skein categories. As a consequence, we we show that the skein module associated to the CKM spider category is isomorphic to the one associated to Sikora's spider category (Theorem \ref{bigckmskeinmodequiv}). 


\subsection{Acknowledgements:} Part of this work was done during the author's PhD dissertation work. The author is grateful to his PhD advisor, Charles Frohman for many helpful discussions and guidance. The author would also like to thank Thang L\^{e} for helpful discussions, and the anonymous referees for many useful suggestions and comments on the earlier version of the paper.

\section{Preliminaries}\label{prelim}
\subsection{Coefficients}  Let $R$ be an integral domain containing $1$, and suppose that $q\in R$ is a unit.  The quantized integers in $R$ are defined to be the sums 
\be [k]=\sum_{i=0}^{k-1}q^{-k+1+2i}.\ee   The quantized factorials are defined recursively by $[0]!=1$ and $[n]!=[n][n-1]!$, and the quantum binomial coefficients
\be \begin{bmatrix} n \\ k\end{bmatrix} =\frac{[n]!}{[k]![n-k]!}.\ee

We will assume that we are working over a ring $R$ having a unit $q^{1/n}$ so that if the category is associated with $sl_n$ then the quantum integers $[1],\ldots,[n]$ are also units.

\subsection{Categories}\label{prelimcats11}
A \textit{pivotal monoidal category}, $\C$, is a rigid monoidal category such that there exist a collection of isomorphisms (a pivotal structure) $a_X: X \xrightarrow{\sim} X^{**}$, natural in $X$, and satisfying $a_{X\otimes Y}=a_X \otimes a_Y$ for all objects $X,Y$ in $\C$. 

For a pivotal monoidal category, $\C$, an object $X \in \C$, and morphisms $f\in \End(X)$, we define the \textit{left} and \textit{right quantum traces}, $Tr_l(f)$ and $Tr_r(f)$ respectively as follows (see \cite{Tur} for more details):
\begin{subequations}
\begin{equation}
    Tr_l(f) = \ev_{X}\circ (\id_{X^*} \otimes f) \circ \coev'_X \in \End(\mathbf{1})
\end{equation}
\begin{equation}
   Tr_r(f) =\ev'_{X}\circ (f \otimes \id_{X^*})\circ \coev_{X} \in \End(\mathbf{1}) 
\end{equation}
\end{subequations}
where (co)$\ev_X$ and (co)$\ev'_X$ are left and right \textit{(co)evaluations}, respectively defined as:
\begin{align*}
    \ev_X:X^*\otimes X \to \mathbf{1} \hspace{1cm} \coev_X: \mathbf{1}\to X \otimes X^*\\
    \ev'_X:X \otimes {}^*X\to \mathbf{1} \hspace{1cm} \coev'_X: \mathbf{1}\to {}^*X \otimes X
\end{align*}

Further, $\C$ is a \textit{spherical monoidal category} if it is a pivotal category such that the left and right quantum traces are the same. In a spherical monoidal category, the \textit{quantum dimension} $d_X$ of an object $X$ is defined to be the quantum trace of identity, $\id_X$. Further, note that $d_X=d_{X^*}$.

A \textit{braided monoidal category} $\C$ is a monoidal category such that there exist a collection of natural isomorphisms (\textit{braid isomorphisms}) $\beta_{X,Y}: X \otimes Y \to Y \otimes X$ for any pair of objects $X,Y\in \C$ that are compatible with the associativity isomorphisms. This compatibility with the associativity isomorphisms in the monoidal category is ensured by the hexagon axiom that the braid isomorphisms satisfy.
We refer the reader to \cite{Tur} for more details. 

We work with spherical braided (\textit{ribbon}) categories via generators and relations. The generators are diagrams carrying labels, where the labels represent irreducible modules over some semisimple Lie algebra. The diagrams represent an element (a vector) in the morphism space (a vector space) of the corresponding category of 
 representations of the Lie algebra. Further, each diagram is considered up to regular isotopy.
In the absence of relations this is called the \textit{free spider category} (on whatever the generators are). The operations are given by (as defined in \cite{kup1}) the following:\\
\textit{Join}: This operation simply allows one to tensor two diagrams (morphisms) by horizontally concatenating.  \\
\textit{Stitch}: For any diagram in $\Hom(A,B)$, this means composing with an evaluation (or coevaluation) to attach a cap or a cup. So, as an example, a stitch could send $\Hom(A\otimes B,C)$ to $\Hom(A\otimes B \otimes B^*, C) \cong \Hom(A, C)$.\\
\textit{Rotation}: This allows one to apply a cyclic permutation on the tensor factors (up to sign). Diagrammatically, this amounts to attaching a cap and a cup to rotate the diagram.
\begin{center}
     \def\svgwidth{8.5cm}
\begingroup%
  \makeatletter%
  \providecommand\color[2][]{%
    \errmessage{(Inkscape) Color is used for the text in Inkscape, but the package 'color.sty' is not loaded}%
    \renewcommand\color[2][]{}%
  }%
  \providecommand\transparent[1]{%
    \errmessage{(Inkscape) Transparency is used (non-zero) for the text in Inkscape, but the package 'transparent.sty' is not loaded}%
    \renewcommand\transparent[1]{}%
  }%
  \providecommand\rotatebox[2]{#2}%
  \newcommand*\fsize{\dimexpr\f@size pt\relax}%
  \newcommand*\lineheight[1]{\fontsize{\fsize}{#1\fsize}\selectfont}%
  \ifx\svgwidth\undefined%
    \setlength{\unitlength}{1213.60378476bp}%
    \ifx\svgscale\undefined%
      \relax%
    \else%
      \setlength{\unitlength}{\unitlength * \real{\svgscale}}%
    \fi%
  \else%
    \setlength{\unitlength}{\svgwidth}%
  \fi%
  \global\let\svgwidth\undefined%
  \global\let\svgscale\undefined%
  \makeatother%
  \begin{picture}(1,0.17461878)%
    \lineheight{1}%
    \setlength\tabcolsep{0pt}%
    \put(0,0){\includegraphics[width=\unitlength,page=1]{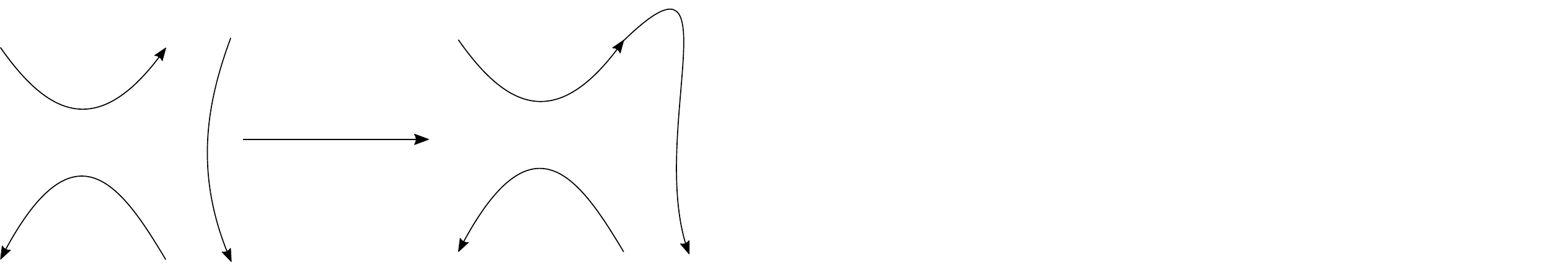}}%
    \put(0.15999326,0.09736884){\color[rgb]{0,0,0}\makebox(0,0)[lt]{\lineheight{1.25}\smash{\begin{tabular}[t]{l}$Stitch$\end{tabular}}}}%
    \put(0,0){\includegraphics[width=\unitlength,page=2]{stitchandrotate.pdf}}%
    \put(0.69731241,0.09269624){\color[rgb]{0,0,0}\makebox(0,0)[lt]{\lineheight{1.25}\smash{\begin{tabular}[t]{l}$Rotate$\end{tabular}}}}%
    \put(0,0){\includegraphics[width=\unitlength,page=3]{stitchandrotate.pdf}}%
  \end{picture}%
\endgroup%

\end{center}
Note that given a spherical tensor category, these (diagrammatic) operations already exist coming from the morphisms in the category.

\section{The braided categories}\label{ckmsikdef}

\subsection{The CKM braided category}\label{ckmcategory1}

As in \cite{ckm}, let $\mathcal{R}ep(SL_n)$ be the category of $U_q(\mathfrak{sl}_n)$-modules generated by tensor products of the fundamental representations. This is a \textit{braided spherical} monoidal category which is a full subcategory of the category of representations of $U_q(\mathfrak{sl}_n)$ where all the morphisms are generated by the wedge product and a version of its adjoint that embeds $\Lambda^{k+l}\mathbb{C}^n$ into $\Lambda^k \CC^n \otimes  \Lambda^l \mathbb{C}^n$:
\begin{align}
    \Lambda^k \CC^n \otimes  \Lambda^l \CC^n \to  \Lambda^{k+l} \CC^n \text{\hspace{1mm} and \hspace{1mm}} \Lambda^{k+l} \CC^n \to \Lambda^k \CC^n \otimes  \Lambda^l \CC^n 
\end{align}
\\
\noindent \textit{The free spider category} $\mathcal{F}Sp(SL_n)$: Define the free spider category to be freely generated by the planar diagrams for morphisms in $\mathcal{R}ep(SL_n)$ as shown below. Objects in $\mathcal{F}Sp(SL_n)$ are subsequences of $\{1^{\pm}, \ldots , (n-1)^{\pm}\}$, where `+' denotes an arrow going upward and `-' denotes an arrow pointing downward. The morphisms are generated by the following:
\begin{figure}[H]
    \centering
  \includegraphics[scale=.5]{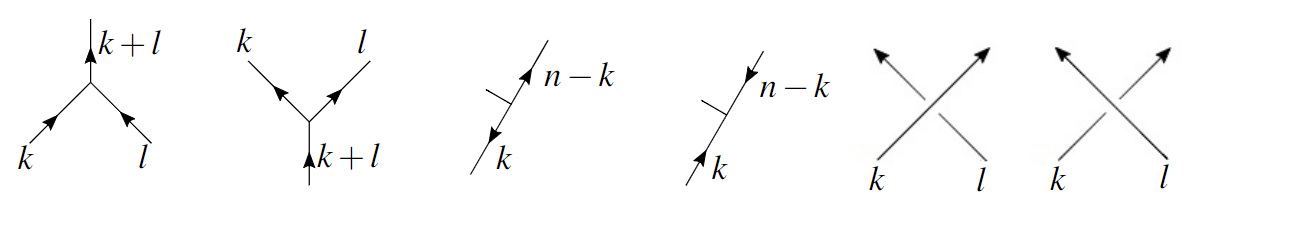}
    \caption{Generators for $\mathcal{F}Sp(SL_n)$}
    \label{generatorsfsp12}
\end{figure}

The spider category, $\mathcal{S}p(SL_n)$, is the quotient of $\mathcal{F}Sp(SL_n)$ by the following relations (together with their mirror reflections and arrow reversals) \cite{ckm}: 
\begin{center}
 \begin{align}
    \vcenter{\includegraphics[scale=.5]{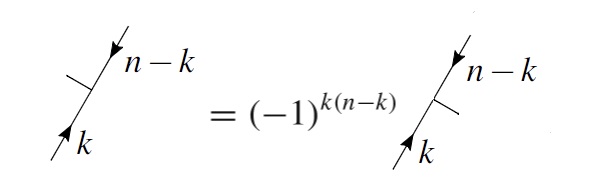}}\label{tagswitchckm}\\
    \vcenter{\includegraphics[scale=.6]{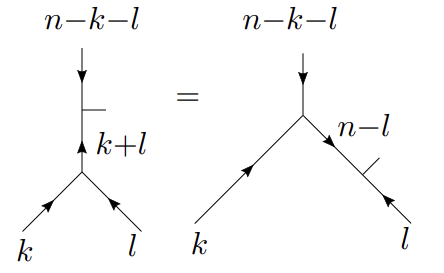}}\label{movetag1ckm}\\
    \vcenter{\includegraphics[scale=.6]{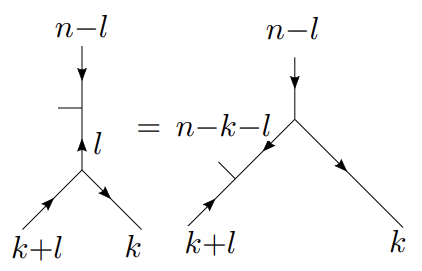}}\label{movetagckm}\\
    \vcenter{\includegraphics[scale=.6]{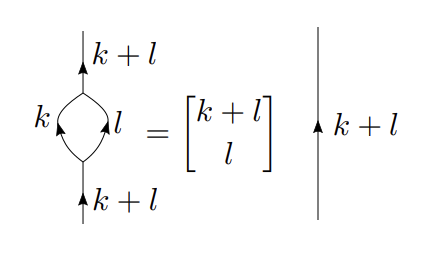}}\label{ckmb1}\\
    \vcenter{\includegraphics[scale=.6]{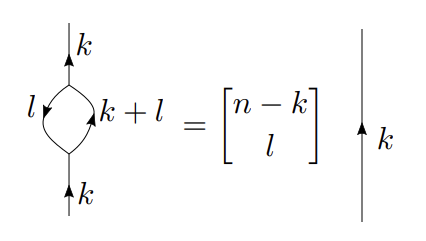}}\label{ckmb2}\\
    \vcenter{\includegraphics[scale=.7]{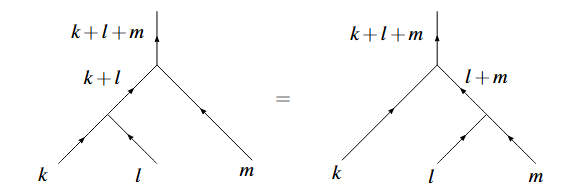}}\label{6jckm1}\\
    \vcenter{\includegraphics[scale=.6]{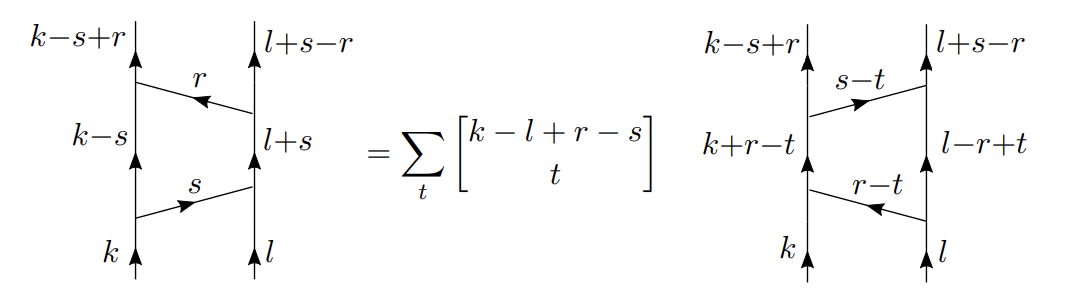}} \label{boxrelationsckm}\\
   \vcenter{ \includegraphics[scale=.6]{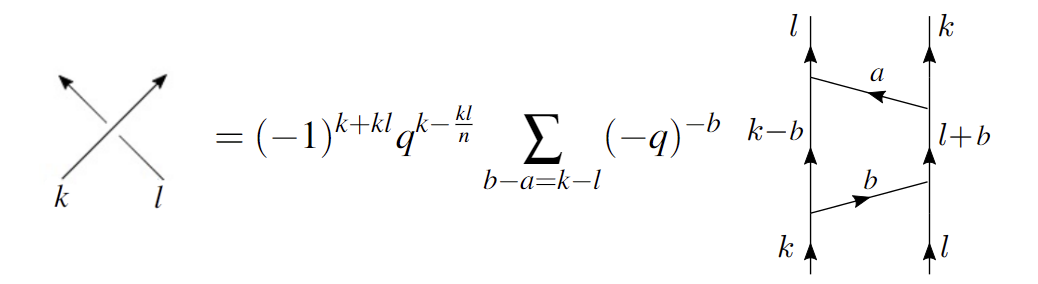}}\label{braiddecomp}\\
    \vcenter{\includegraphics[scale=.4]{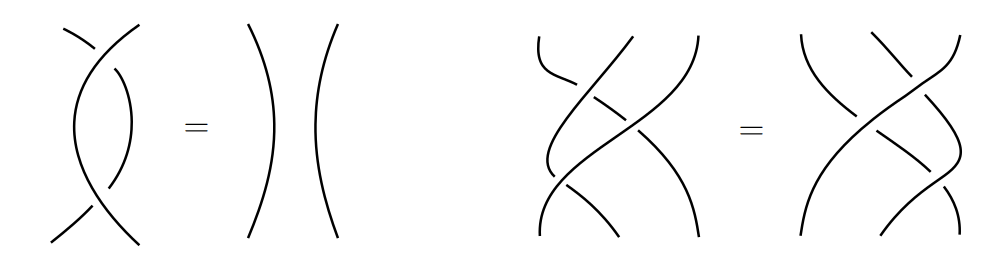}
   \includegraphics[scale=.27]{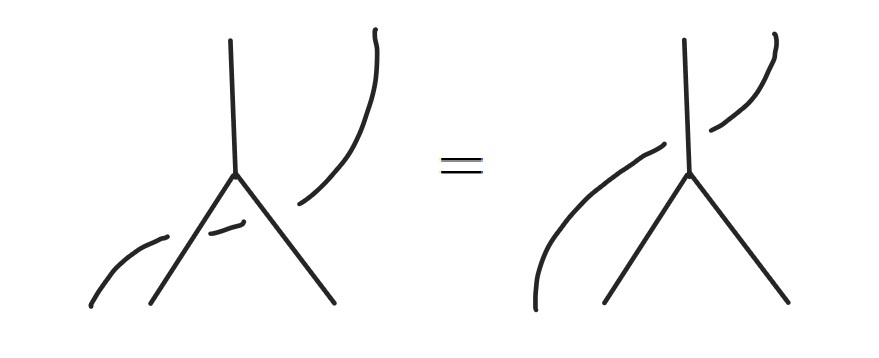}}\label{reidckm1}
\end{align}
   \end{center}

Note, for a negative crossing, the RHS is obtained by changing $q$ to $q^{-1}$ in the relation (\ref{braiddecomp}) above. Further, the Reidemeister relations have diagrams with unlabeled and unoriented edges as those hold for any admissible labels and orientation of edges, along with corresponding diagrams where undercrossings are changed to overcrossings and vice-versa. We will refer to the relations (\ref{tagswitchckm}, \ref{movetag1ckm}, \ref{movetagckm}) as ``tag relations", (\ref{ckmb1}) and (\ref{ckmb2}) as the ``bubble relations", (\ref{6jckm1}) as ``the 6j move" or ``the 6j relation", (\ref{boxrelationsckm}) as ``the box relation" and (\ref{braiddecomp}) as ``the braid relation". 
\\
\begin{remark}
 Even though in \cite{ckm} the authors work over $\mathbb{C}(q)$, these web relations are known to provide a complete set of relations over $\mathbb{Z}[q,q^{-1}]$, for example see \cite{be} for a proof of this fact. 


\end{remark}



\subsection{The Sikora category} \label{sikoracat}
The paper \cite{sik} is not couched in category theoretic terms. However, there is a natural way to describe his work in a category theoretic setting, which has been explicitly carried out in \cite{lesik}. In this paper, we present his work in \cite{sik} in terms of a quotient of a free spider category.
Consider a free spider category with objects sequences in $\{\pm\}$ where again `$+$' means edges going up and `$-$' means edges going down and the morphisms are generated by:
\begin{center}
\def\svgwidth{6cm}
\begingroup%
  \makeatletter%
  \providecommand\color[2][]{%
    \errmessage{(Inkscape) Color is used for the text in Inkscape, but the package 'color.sty' is not loaded}%
    \renewcommand\color[2][]{}%
  }%
  \providecommand\transparent[1]{%
    \errmessage{(Inkscape) Transparency is used (non-zero) for the text in Inkscape, but the package 'transparent.sty' is not loaded}%
    \renewcommand\transparent[1]{}%
  }%
  \providecommand\rotatebox[2]{#2}%
  \newcommand*\fsize{\dimexpr\f@size pt\relax}%
  \newcommand*\lineheight[1]{\fontsize{\fsize}{#1\fsize}\selectfont}%
  \ifx\svgwidth\undefined%
    \setlength{\unitlength}{501.25667061bp}%
    \ifx\svgscale\undefined%
      \relax%
    \else%
      \setlength{\unitlength}{\unitlength * \real{\svgscale}}%
    \fi%
  \else%
    \setlength{\unitlength}{\svgwidth}%
  \fi%
  \global\let\svgwidth\undefined%
  \global\let\svgscale\undefined%
  \makeatother%
  \begin{picture}(1,0.26474416)%
    \lineheight{1}%
    \setlength\tabcolsep{0pt}%
    \put(0,0){\includegraphics[width=\unitlength,page=1]{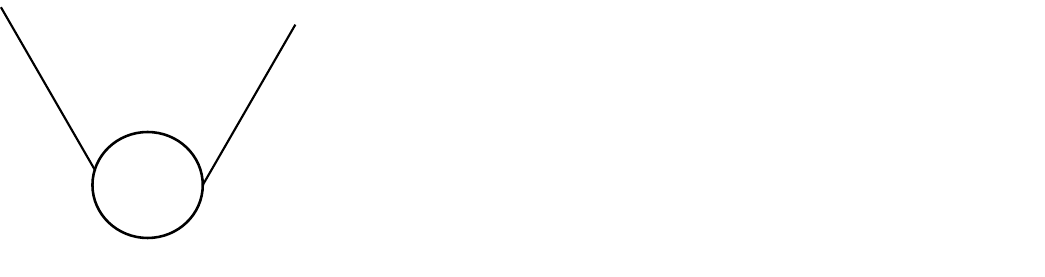}}%
    \put(0.1122939,0.15582709){\color[rgb]{0,0,0}\makebox(0,0)[lt]{\lineheight{1.25}\smash{\begin{tabular}[t]{l}$\cdots$\end{tabular}}}}%
    \put(0,0){\includegraphics[width=\unitlength,page=2]{sikoragen.pdf}}%
  \end{picture}%
\endgroup%
\\
\end{center}
Note, the leftmost vertex is either a source or a sink and each morphism in this category is represented by a $n-$valent ribbon graph considered up to regular isotopy. Just as before, this is a braided spherical category which is a full subcategory of the category of $U_q(\mathfrak{sl}_n)$-modules whose objects are monoidal product of copies of the standard representation and its dual.
We call this category $\Tilde{\mathcal{S}}$ if the morphisms satisfy the following relations:
\begin{align}
\includegraphics[scale=.7]{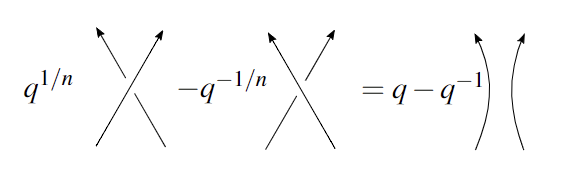}\label{homflysik}\\
\includegraphics[scale=.7]{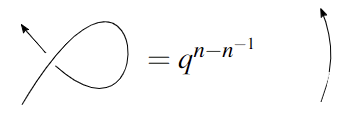}\label{twistsik}\\
\includegraphics[scale=.7]{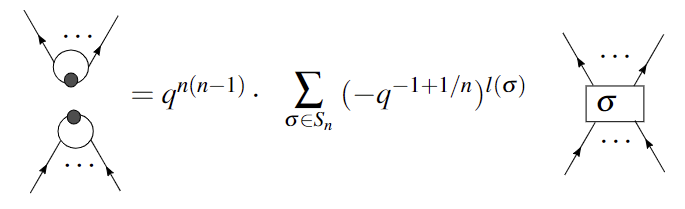}\label{sourcesink}\\
\includegraphics[scale=.7]{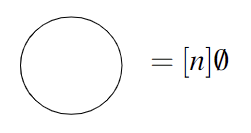}\label{unknotsik}
\end{align}

Here, for any permutation, $\sigma \in S_n$, $l(\sigma)$ denotes its length:
\[l(\sigma) = \#\{(i,j): 1\leq i < j \leq n,  \sigma(i)>\sigma(j)\}
\]
The diagram with $\sigma$ as a coupon represents a positive braid with $l(\sigma)$ crossings representing $\sigma$. Note that along with these relations, this is also the category $\mathfrak{S}_n^b$ in \cite{lesik}.
Check \cite{sik, lesik} for more details.

\subsection{The MOY category}\label{originalmoycat1}
There is no attempt to establish a complete diagrammatic presentation of a category in \cite{moy}. However, the works \cite{mor,ckm} recapitulate the generators and relators given in \cite{moy}. Define a category $\mathcal{MOY}(SL_n)$ to be a spider category with objects sequences in $\{1^\pm, \cdots, (n-1)^\pm, n^\pm\}$ and morphisms generated by the trivalent vertices and crossings given in Figure \ref{generatorsfsp12}. The relators are given by all the relations in $\mathcal{S}p(SL_n)$ except the tag relations. Note that the conventions regarding the objects are the same as in $\mathcal{S}p(SL_n)$, however, in $\mathcal{MOY}(SL_n)$ edges with label $n$ are allowed and there are no tags.


\section{Understanding the CKM box relations}\label{boxes}
\noindent  In this section, we derive the box relations (\ref{boxrelationsckm}) from Reidemeister invariance (\ref{reidckm1}), bubble (\ref{ckmb1} and \ref{ckmb2}) and 6j relations (\ref{6jckm1}). Some results in this section were also observed in \cite{big}.

\begin{lemma}\label{recursion1}
The following equation is a consequence of the bubble (\ref{ckmb1}, \ref{ckmb2}), $6j$ (\ref{6jckm1}), braid (\ref{braiddecomp}) and the Reidemeister (\ref{reidckm1}) relations.
\begin{align*}
\includegraphics[scale=.8]{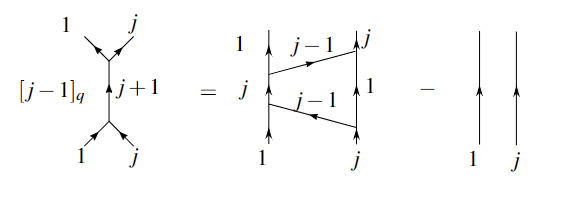}
\end{align*}
\end{lemma}

\begin{proof}
Recalling the decomposition of braid isomorphism in terms of webs given in relation (\ref{braiddecomp}) and then applying Reidemeister II invariance gives us the following.\\ 
\begin{center}
 \includegraphics[scale=.6]{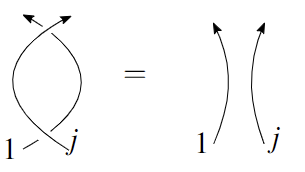}\\
\includegraphics[scale=.6]{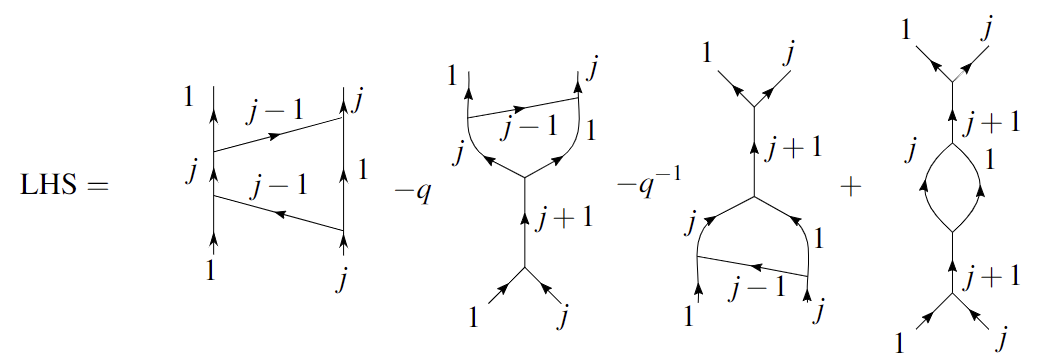}\\
\includegraphics[scale=.6]{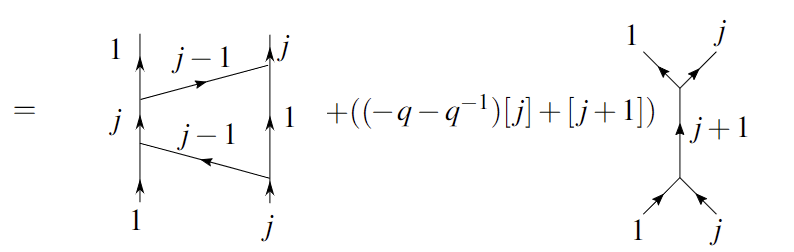}
\end{center}
The claim follows from:
\begin{align*}
    (-q-q^{-1})[j]+[j+1] = -[j-1].
\end{align*}
\end{proof}

\newpage
\begin{lemma}\label{generalizedrecursion1}
The following equation is a consequence of the bubble (\ref{ckmb1}, \ref{ckmb2}), the $6j$ (\ref{6jckm1}), braid (\ref{braiddecomp}) and the Reidemeister (\ref{reidckm1}) relations.
\begin{center}
\def\svgwidth{15.5cm}
\begingroup%
  \makeatletter%
  \providecommand\color[2][]{%
    \errmessage{(Inkscape) Color is used for the text in Inkscape, but the package 'color.sty' is not loaded}%
    \renewcommand\color[2][]{}%
  }%
  \providecommand\transparent[1]{%
    \errmessage{(Inkscape) Transparency is used (non-zero) for the text in Inkscape, but the package 'transparent.sty' is not loaded}%
    \renewcommand\transparent[1]{}%
  }%
  \providecommand\rotatebox[2]{#2}%
  \newcommand*\fsize{\dimexpr\f@size pt\relax}%
  \newcommand*\lineheight[1]{\fontsize{\fsize}{#1\fsize}\selectfont}%
  \ifx\svgwidth\undefined%
    \setlength{\unitlength}{1013.10571866bp}%
    \ifx\svgscale\undefined%
      \relax%
    \else%
      \setlength{\unitlength}{\unitlength * \real{\svgscale}}%
    \fi%
  \else%
    \setlength{\unitlength}{\svgwidth}%
  \fi%
  \global\let\svgwidth\undefined%
  \global\let\svgscale\undefined%
  \makeatother%
  \begin{picture}(1,0.20902299)%
    \lineheight{1}%
    \setlength\tabcolsep{0pt}%
    \put(0,0){\includegraphics[width=\unitlength,page=1]{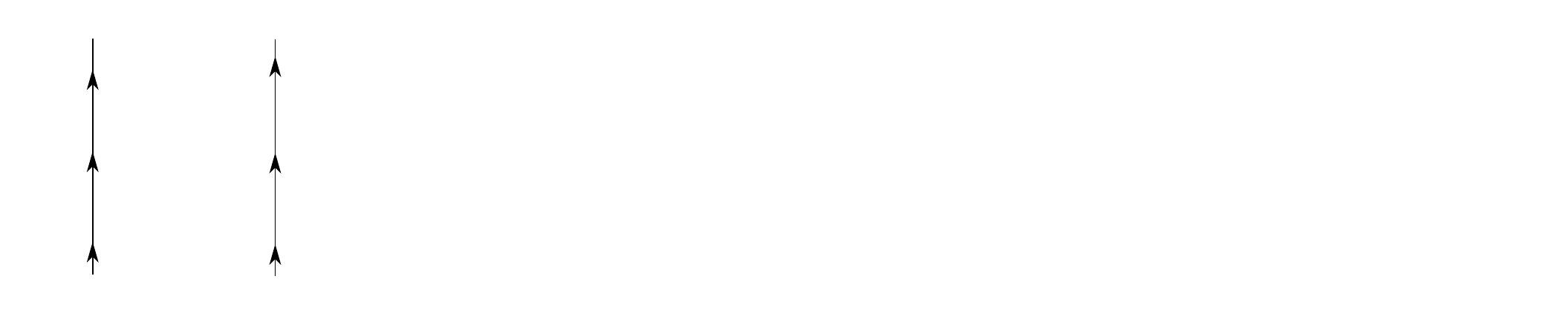}}%
    \put(0.27011263,0.09870761){\color[rgb]{0,0,0}\makebox(0,0)[lt]{\lineheight{1.25}\smash{\begin{tabular}[t]{l}$=\begin{bmatrix} j-1 \\ k-1\end{bmatrix}$\end{tabular}}}}%
    \put(0.49050248,0.10377011){\color[rgb]{0,0,0}\makebox(0,0)[lt]{\lineheight{1.25}\smash{\begin{tabular}[t]{l}$+\begin{bmatrix} j-1 \\ k\end{bmatrix}$\end{tabular}}}}%
    \put(0,0){\includegraphics[width=\unitlength,page=2]{Lemma2.pdf}}%
    \put(0.40722454,0.10672336){\color[rgb]{0,0,0}\makebox(0,0)[lt]{\lineheight{1.25}\smash{\begin{tabular}[t]{l}$j+1$\end{tabular}}}}%
    \put(0.33458549,0.18832253){\color[rgb]{0,0,0}\makebox(0,0)[lt]{\lineheight{1.25}\smash{\begin{tabular}[t]{l}$1$\end{tabular}}}}%
    \put(0.42355961,0.19232826){\color[rgb]{0,0,0}\makebox(0,0)[lt]{\lineheight{1.25}\smash{\begin{tabular}[t]{l}$j$\end{tabular}}}}%
    \put(0.35490997,0.01426633){\color[rgb]{0,0,0}\makebox(0,0)[lt]{\lineheight{1.25}\smash{\begin{tabular}[t]{l}$1$\end{tabular}}}}%
    \put(0.42518668,0.0151558){\color[rgb]{0,0,0}\makebox(0,0)[lt]{\lineheight{1.25}\smash{\begin{tabular}[t]{l}$j$\end{tabular}}}}%
    \put(0.04266804,0.00537036){\color[rgb]{0,0,0}\makebox(0,0)[lt]{\lineheight{1.25}\smash{\begin{tabular}[t]{l}$1$\end{tabular}}}}%
    \put(-0.00172595,0.0950182){\color[rgb]{0,0,0}\makebox(0,0)[lt]{\lineheight{1.25}\smash{\begin{tabular}[t]{l}$k+1$\end{tabular}}}}%
    \put(0.0252815,0.16149801){\color[rgb]{0,0,0}\makebox(0,0)[lt]{\lineheight{1.25}\smash{\begin{tabular}[t]{l}$1$\end{tabular}}}}%
    \put(0.08656768,0.1552656){\color[rgb]{0,0,0}\makebox(0,0)[lt]{\lineheight{1.25}\smash{\begin{tabular}[t]{l}$k$\end{tabular}}}}%
    \put(0.18109379,0.16565309){\color[rgb]{0,0,0}\makebox(0,0)[lt]{\lineheight{1.25}\smash{\begin{tabular}[t]{l}$j$\end{tabular}}}}%
    \put(0.1852488,0.09917316){\color[rgb]{0,0,0}\makebox(0,0)[lt]{\lineheight{1.25}\smash{\begin{tabular}[t]{l}$j-k$\end{tabular}}}}%
    \put(0.09072267,0.08255317){\color[rgb]{0,0,0}\makebox(0,0)[lt]{\lineheight{1.25}\smash{\begin{tabular}[t]{l}$k$\end{tabular}}}}%
    \put(0.16711982,0.0044809){\color[rgb]{0,0,0}\makebox(0,0)[lt]{\lineheight{1.25}\smash{\begin{tabular}[t]{l}$j$\end{tabular}}}}%
    \put(0,0){\includegraphics[width=\unitlength,page=3]{Lemma2.pdf}}%
    \put(0.58682422,0.00823114){\color[rgb]{0,0,0}\makebox(0,0)[lt]{\lineheight{1.25}\smash{\begin{tabular}[t]{l}$1$\end{tabular}}}}%
    \put(0.64375733,0.00734168){\color[rgb]{0,0,0}\makebox(0,0)[lt]{\lineheight{1.25}\smash{\begin{tabular}[t]{l}$j$\end{tabular}}}}%
    \put(0,0){\includegraphics[width=\unitlength,page=4]{Lemma2.pdf}}%
  \end{picture}%
\endgroup%

\end{center}
\end{lemma}

\begin{proof}
This follows by induction on $j-k$. \\
\textit{Base case:} when $j-k=1$, this follows from Lemma \ref{recursion1} above. Let this hold for all values up to $j-k=m$. Now, apply Lemma \ref{recursion1} with $j=k$ on LHS to get:

    \includegraphics[scale=.65]{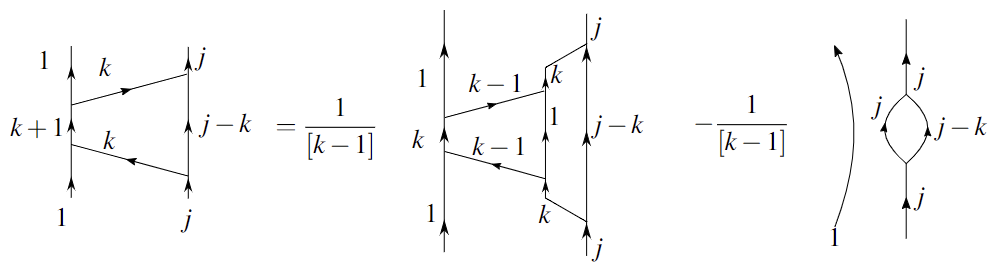}\\
    \hspace*{4cm}\includegraphics[scale=.65]{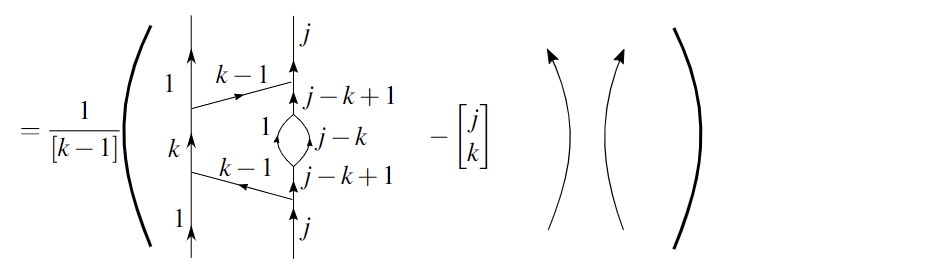}\\
    \hspace*{4cm}\includegraphics[scale=.65]{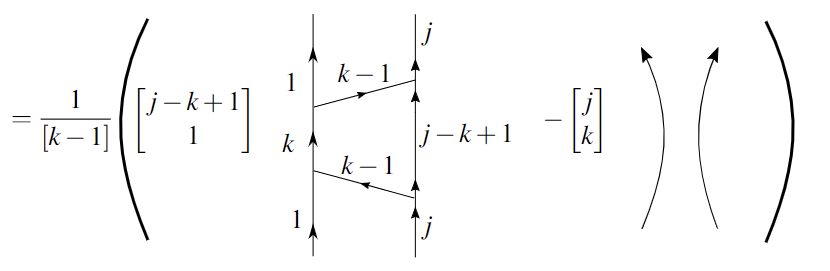}\\

Now the proof follows by Induction hypothesis on the box obtained on the last line by noticing that $j-k+1=m \implies j-k=m-1$.

\end{proof}

\newpage

\begin{lemma}\label{boxrelationbasecase}
As a consequence of the bubble (\ref{ckmb1}, \ref{ckmb2}), the $6j$ (\ref{6jckm1}), braid (\ref{braiddecomp}) and the Reidemeister (\ref{reidckm1}) relations, we obtain:\\
\begin{center}
    \def\svgwidth{11cm}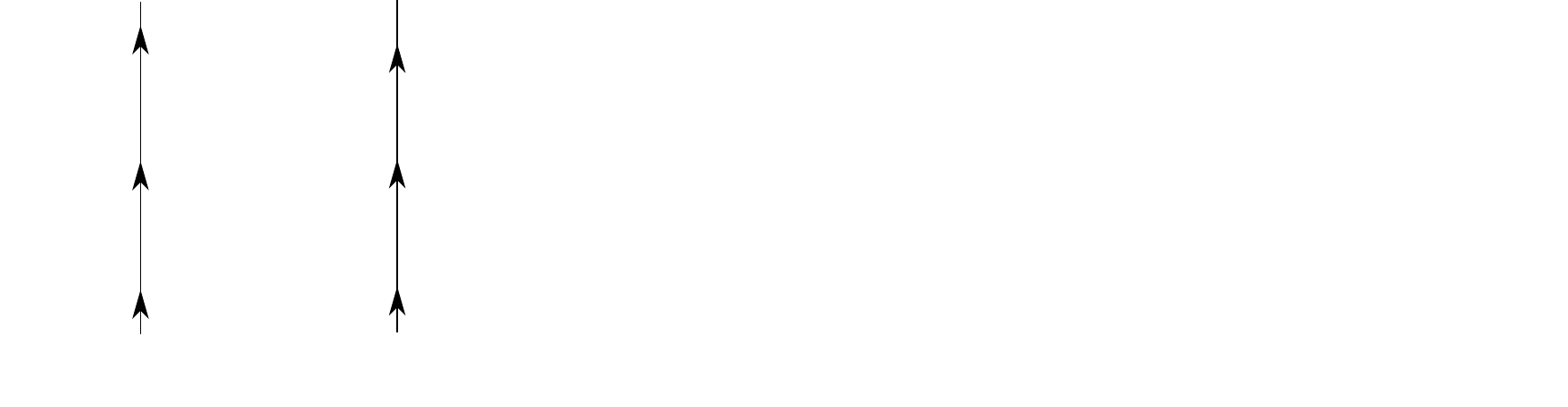
\end{center}
\end{lemma}

\begin{proof}
As before, start from the box on the RHS to get \\
\includegraphics[scale=.7]{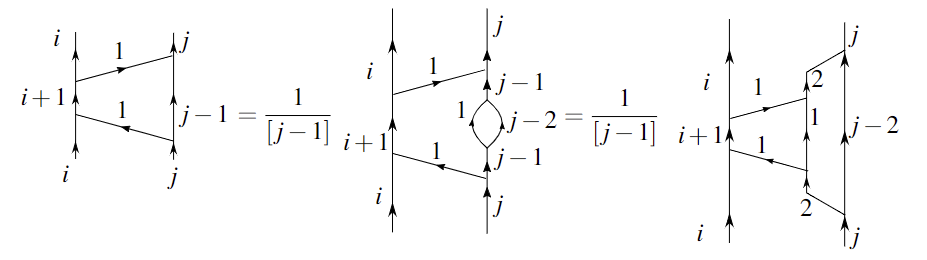}\\
\hspace*{3.45cm}\includegraphics[scale=.7]{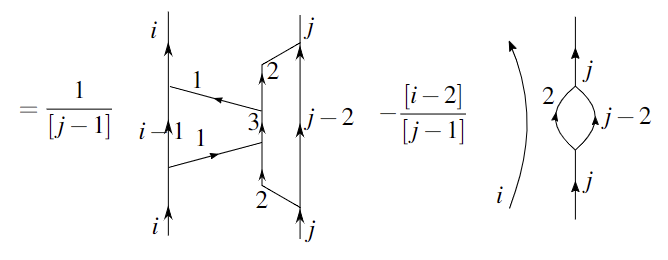}\\
\hspace*{3.45cm}\includegraphics[scale=.7]{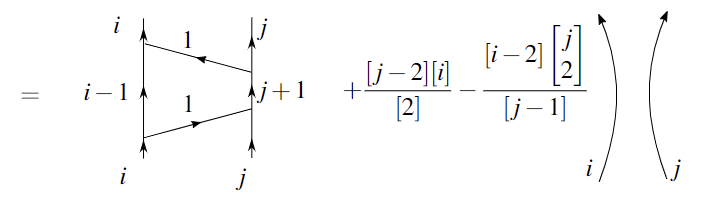}\\
\hspace*{3.45cm}\includegraphics[scale=.6]{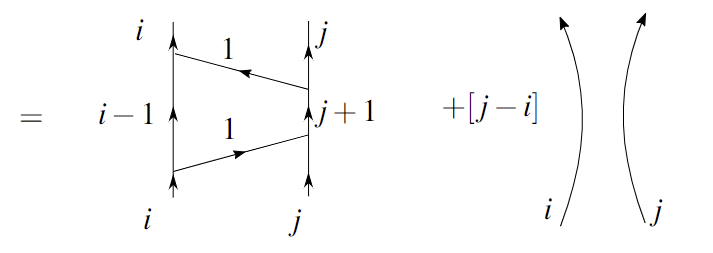}\\

Above, we use Lemma \ref{recursion1} on the third equality and Lemma \ref{generalizedrecursion1} to obtain the fourth equality.

\end{proof}

\newpage

\begin{theorem}\label{allckmrelationtheorem1}
The list of relations in $\mathcal{S}p(SL_n)$ (c.f. Section \ref{ckmcategory1}), not including the tag relations (\ref{tagswitchckm}, \ref{movetag1ckm}, and \ref{movetagckm}), follow from the Reidemeister invariance, the bubble (\ref{ckmb1} and \ref{ckmb2}), braid (\ref{braiddecomp}) and the $6j$ relations (\ref{6jckm1}).
\end{theorem}
\begin{proof}
Following the remark in Section 2.2 in \cite{ckm}, Lemma \ref{boxrelationbasecase} gives us all the box relations. Since Lemma \ref{boxrelationbasecase} follows from assuming the Reidemeister II invariance, the bubble and the $6j$ relations, this gives us the result.
\end{proof}




\section{The full subcategory based on the standard representation}\label{fullsubcat}
Throughout this section, whenever $q$ is a root of unity, it is assumed that the quantum integers $[1],\ldots,[n]$ are invertible. In this section, our goal is to show that the relations in the full subcategory of $\mathcal{MOY}(SL_n)$ with objects sequences in $\{1^\pm\}$ is completely characterized by the specialized HOMFLYPT relation discussed below.
Let $a, v, s$ and $s-s^{-1}$ be invertible elements in our integral domain $R$.
\begin{figure}[H]
    \centering
    \includegraphics[scale=.6]{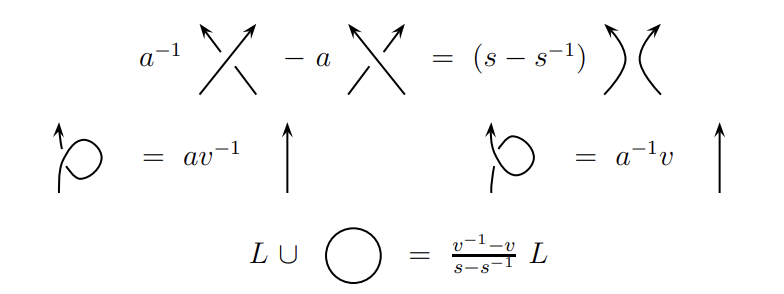}
    \caption{the HOMFLYPT relations}
    \label{blanchethomfly}
\end{figure}
The specialized HOMFLYPT relation is obtained by picking the objects to be sequences in $\{1^{\pm}\}$ and by setting $a=q^{-1/n}$, $s=q$ and $v = q^{-n}$.


\subsection{Categories \(\mathcal{MOY}^i(SL_n)\)} Let $\mathcal{MOY}^i(SL_n)$ denote a quotient category which is a subset of the full subcategory of \(\mathcal{MOY}(SL_n)\) with objects sequences in $\{1^\pm\}$ and morphisms generated by trivalent vertices and crossings that do not contain any edge with a color `$x$' such that $x>i$. 

Let $\mathcal{R}el^i(SL_n)$ be the set of relations in $\mathcal{MOY}^i(SL_n)$. Then $\mathcal{R}el^i(SL_n)$ consists of all the relations in the full subcategory of $\mathcal{MOY}(SL_n)$ that involve the colors less or equal to `$i$'. 
Regarding the braid relations (\ref{braiddecomp}) consisting of crossings with colors `$j$',`$k$' such that and $j+k=i+1$, $\mathcal{R}el^i(SL_n)$ consists of HOMFLYPT-type relations involving the difference of over and under-crossings which equal the sum of diagrams with colors lower than $j+k$. For example, in $\mathcal{R}el^1(SL_n)$, this gives us the specialized HOMFLYPT relation involving the colors `$1$' and in $\mathcal{R}el^2(SL_n)$, this gives a relation involving the difference of crossings between colors `$1$' and `$2$' and so on. For the case of crossings with colors $`j',`k'$ such that $j+k>i+1$, we note that all our diagrams live in the full subcategory which allows for the isotopy relation in $\mathcal{R}el^i(SL_n)$ to drag the $j$ (or the $k$) color across a vertex to obtain crossing relations (possibly more than one) in terms of lower colors.
This way, we write down an equivalent braid relation at each level without involving any higher colors in the summand.


Since $\mathcal{R}el^{i-1}(SL_n) \subseteq \mathcal{R}el^i(SL_n)$, there is a chain of categories defined by restricting the colors given by natural (inclusion) functors $\psi_i$ as follows:
\begin{align*}
    \mathcal{MOY}^1(SL_n) \xhookrightarrow{\psi_1} \mathcal{MOY}^2(SL_n) \xhookrightarrow{\psi_2} \mathcal{MOY}^3(SL_n) \xhookrightarrow{} \cdots \xhookrightarrow{\psi_{i-1}} \mathcal{MOY}^i(SL_n)
\end{align*}
\\
Below, we construct a spherical braided functor, $\phi_i$ from $\mathcal{MOY}^i(SL_n)$ to $\mathcal{MOY}^{i-1}(SL_n)$ explicitly.

\subsection{Constructing $\phi_i$} In order to define the map $\phi_i: \mathcal{MOY}^i(SL_n) \to \mathcal{MOY}^{i-1}(SL_n)$, we start with the following observation which follows from Lemma \ref{generalizedrecursion1}.\\
\begin{equation}\label{i.2a}
 \def\svgwidth{14cm}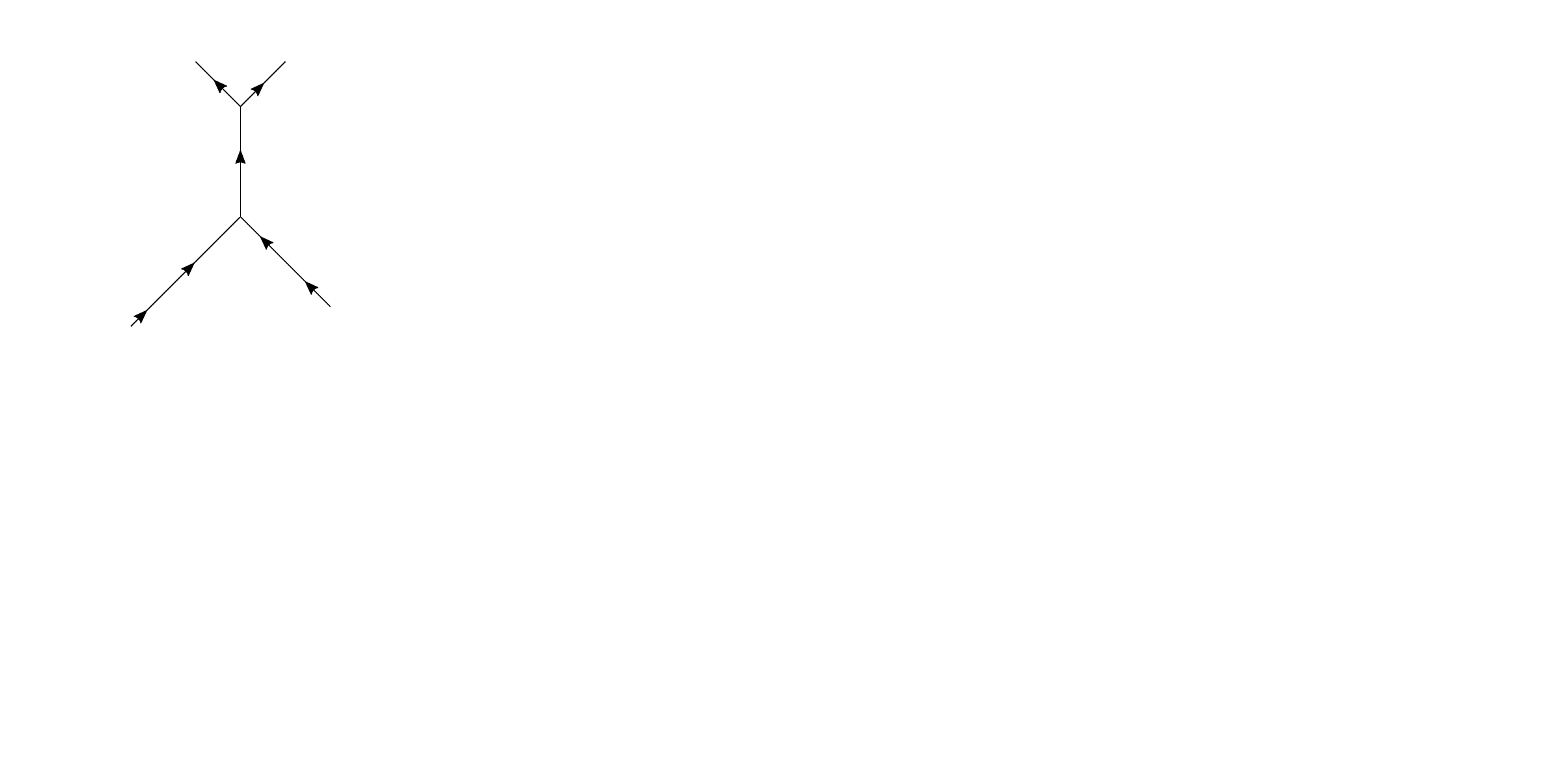
\end{equation}

where, going from the first to second equation, we use the $6j$ move. One can use the idea above to obtain a relation for a generic diagram with the color `$i$' in the middle. 

Thus, define the functor $\phi_i: \mathcal{MOY}^i(SL_n) \to \mathcal{MOY}^{i-1}(SL_n)$ to be identity on $\mathcal{MOY}^{i-1}(SL_n) \subseteq \mathcal{MOY}^i(SL_n)$ and on $\mathcal{MOY}^i(SL_n) \setminus \mathcal{MOY}^{i-1}(SL_n)$, define the image of a morphism with a color `$i$' in its edge as

\begin{equation}\label{i.2}
\def\svgwidth{14.65cm}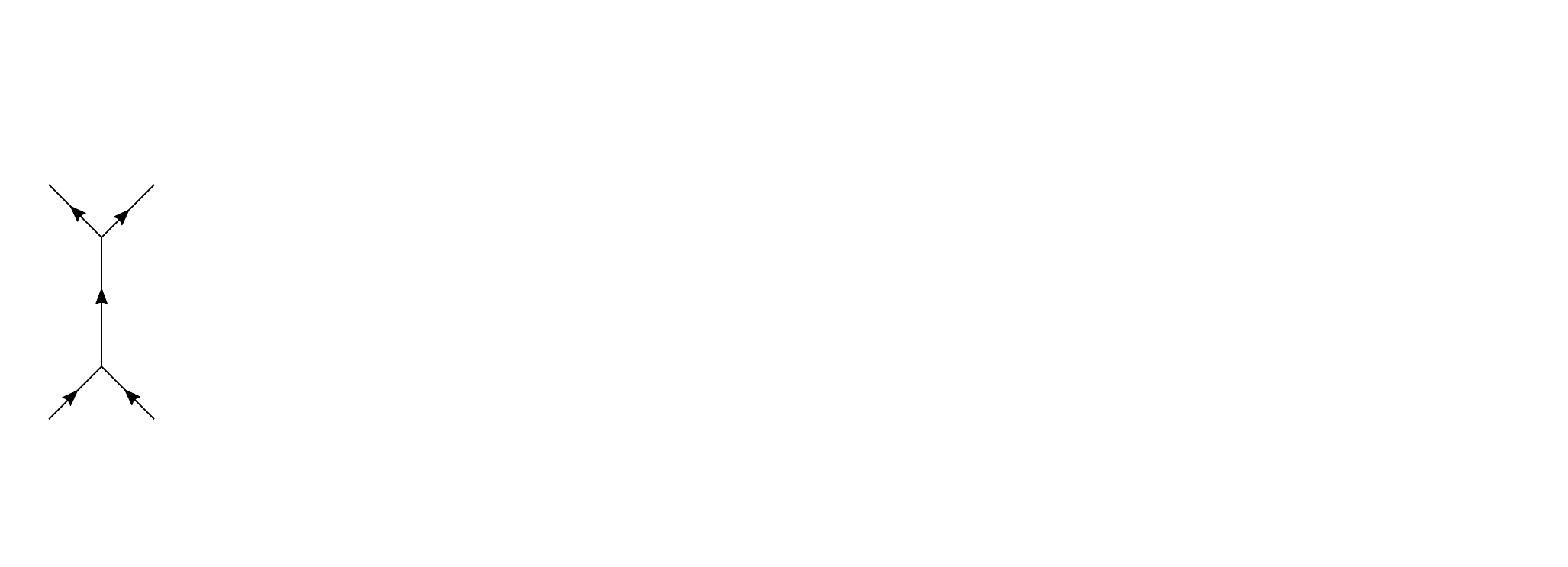
\end{equation}
Note that using our observation in (\ref{i.2a}), one obtains a relation with LHS and RHS corresponding to the diagrams on (\ref{i.2}) which was the reason for defining $\phi_i$ in this manner in (\ref{i.2}) above. Below, we prove that the spherical braided functor $\phi_i$ gives a categorical equivalence. We do this in two steps in order to provide a cleaner exposition. First we consider the case $i\geq 3$, which involves working with trivalent vertices and after that we consider the case of $i=2$.






\newpage

\begin{lemma}\label{faithfulfunctorphi}
For $i>3$, the functor $\phi_i:\mathcal{MOY}^i(SL_n) \to \mathcal{MOY}^{i-1}(SL_n)$ is fully faithful.
\end{lemma}
\begin{proof}
As $\phi_i$ is a map between quotient categories, its well-definedness and faithfulness is established by showing that all the relations $\phi_i(\mathcal{R}el^i(SL_n)) \subseteq \mathcal{R}el^{i-1}(SL_n)$.To that end, recalling from Theorem \ref{allckmrelationtheorem1}, it is enough to check the braid and two Reidemeister relations, the bubble and the $6j$ relations involving the color `$i$' and verify that their images under $\phi_i$ is contained in $\mathcal{R}el^{i-1}(SL_n)$.

\noindent\textit{The $6j$ relation}: In order to show that $\phi_i$ respects the $6j$ relation involving the color `$i$', we need to first show the following equality holds in $\mathcal{MOY}^{i-1}(SL_n)$.

\underline{Claim:}\\
\includegraphics[scale=.55]{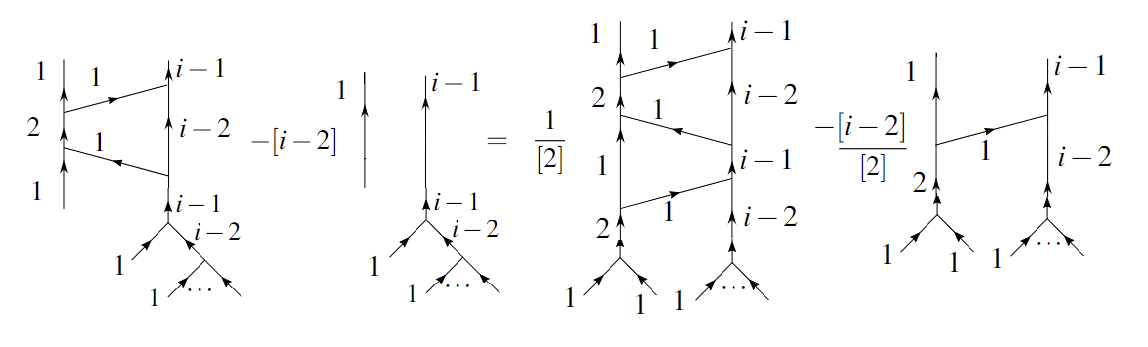}

\begin{proof}[Proof of the claim]
Starting with the box on the LHS, and applying the relation from Lemma \ref{generalizedrecursion1} on the bottom `$i-1$' color gives us\\
\includegraphics[scale=.7]{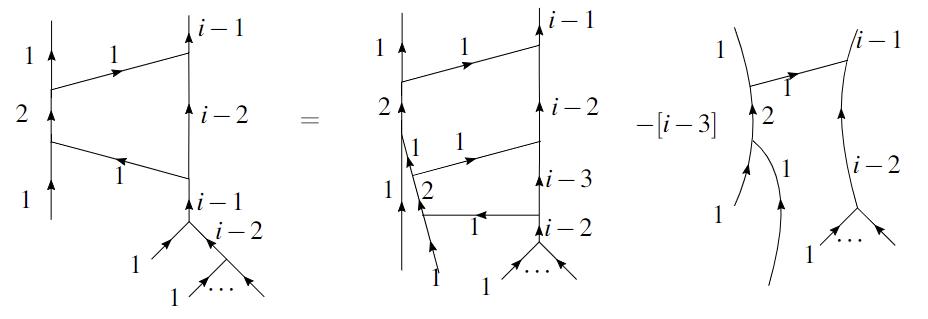}\\
\hspace*{4cm}\includegraphics[scale=.7]{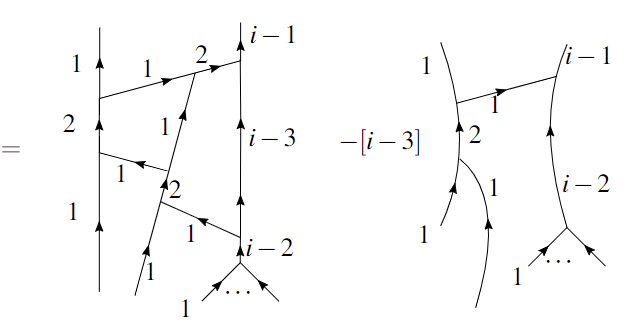}

Now, applying the Lemma \ref{generalizedrecursion1} on the leftmost box above gives us 

\includegraphics[scale=.7]{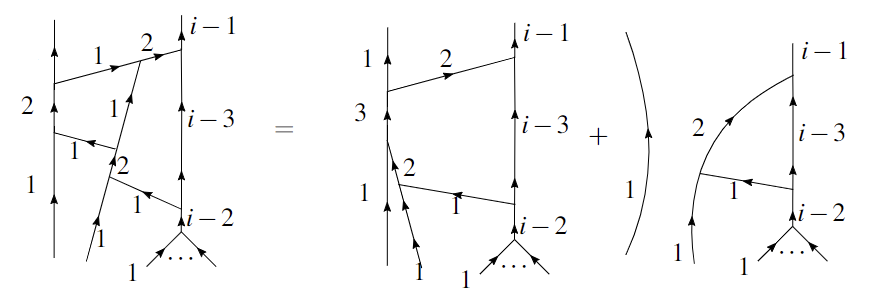}\\

\hspace*{4cm}\includegraphics[scale=.7]{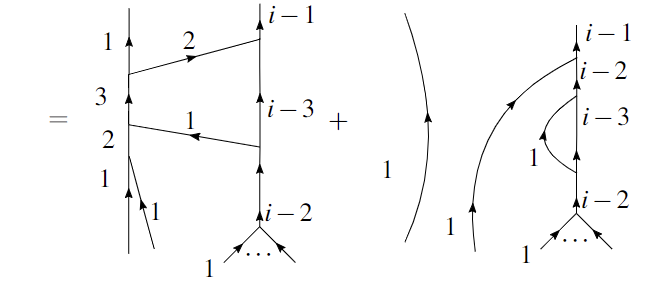}\\

\hspace*{4cm}\includegraphics[scale=.7]{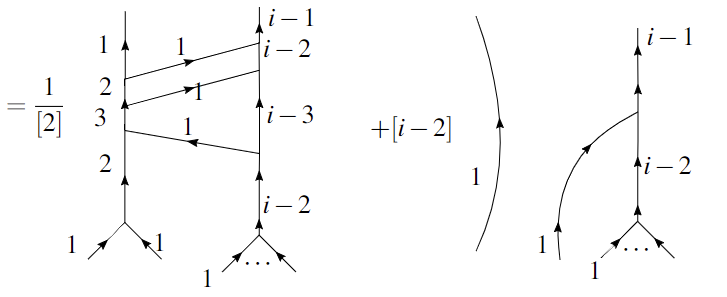}\\
Finally, the claim follows from applying the relation from Lemma \ref{boxrelationbasecase} to the lower box on the left diagram above and combining all the terms that we started with on the LHS. 
\end{proof}
\noindent Note that all the other $6j$ relations involving the color `$i$' follow by applying induction using the method we showed above.

\noindent\textit{Braid and Reidemeister relations}: The braid relations with an `$i$' colored strand in the crossings get mapped to a sum of diagrams with strictly lower colors under $\phi_i$, giving crossings with lower colors as shown in the diagram below. Whence, one can check that these relations are satisfied using braid relations in $\mathcal{R}el^{i-1}(SL_n)$ after applying a sequence of Reidemeister moves to simplify the resolutions to get lower colors in the crossings (as discussed below) when necessary.     

Consider the Reidemeister II relation involving color `$i$' in $\mathcal{R}el^i(SL_n)$. In order to see that its image under $\phi_i$ is contained in $\mathcal{R}el^{i-1}(SL_n)$ we proceed as shown below. Note that by $D_\alpha$, we mean the linear combination of trivalent oriented graphs obtained as the image of an edge with color `$i$' under $\phi_i$ shown in (\ref{i.2}) and $\gamma_\alpha$ are the corresponding coefficients. As each edge in $D_\alpha$ has label strictly less than `$i$', using a sequence of Reidemeister II's and III's coming from $\mathcal{R}el^{i-1}(SL_n)$, one obtains the result. 
\begin{center}
\includegraphics[scale=.65]{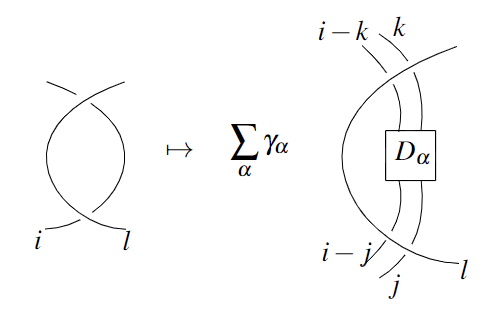}
\end{center}

Similarly, now consider the Reidemeister III relation involving the color `$i$'. In this case, one obtains the result by using the Kauffman Trick to observe that its image under $\phi_i$ is contained in $\mathcal{R}el^{i-1}(SL_n)$ by noticing that each edge in $D_\alpha$ has label strictly less than `$i$' as shown below:

\begin{center}
\includegraphics[scale=.55]{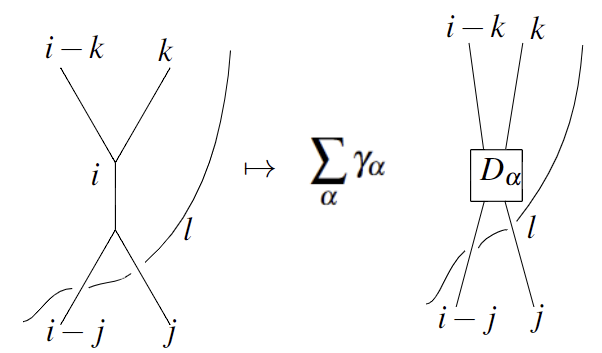}
\end{center}


\noindent\textit{The bubble relation I:}
Take the general relation \ref{i.2} and stack them to get:\\
\begin{align}\label{crazybubble1}
  \def\svgwidth{18.5cm}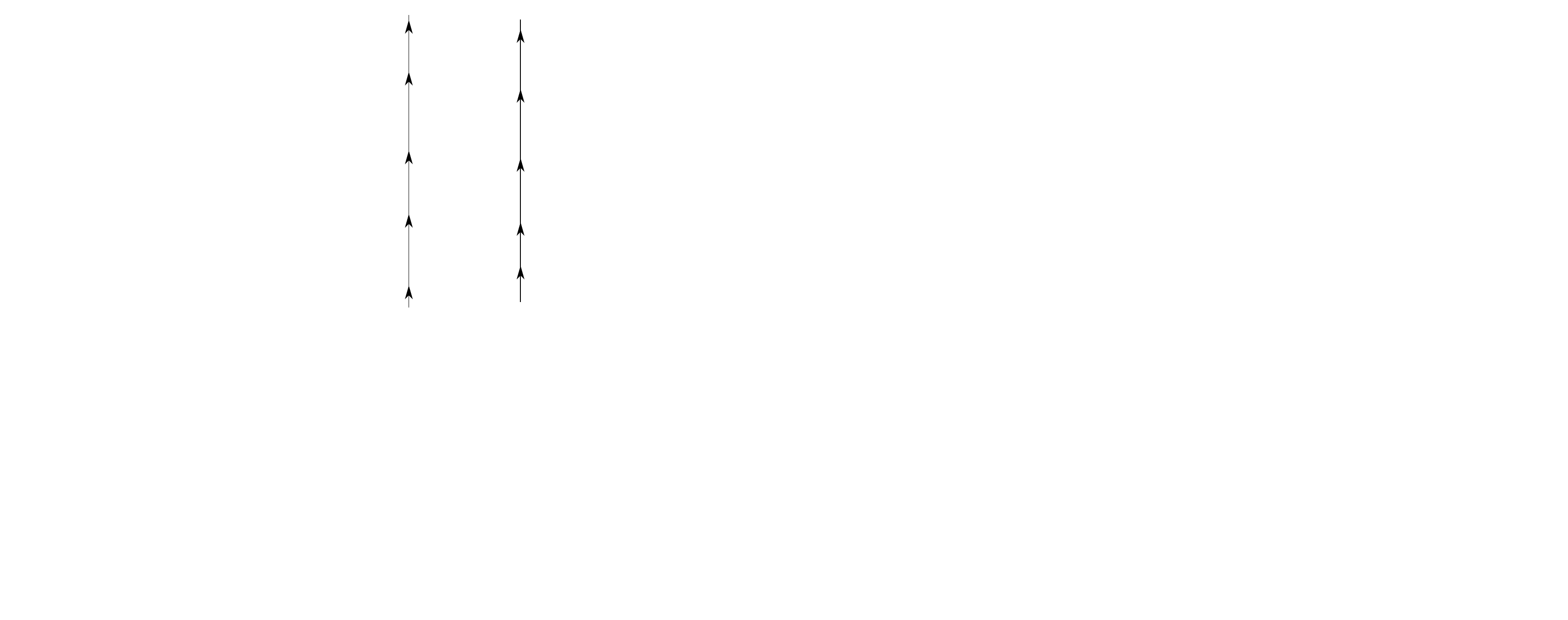
\end{align}
\\

where, $A=\dfrac{1}{[j][k]}$ and $B = \dfrac{[i-2]}{[j][k]}$.\\

We analyze the first box from the RHS above. Start by applying the relation from Lemma (\ref{recursion1}) to the left edge colored `$j$' in the following manner.\\

\includegraphics[scale=.6]{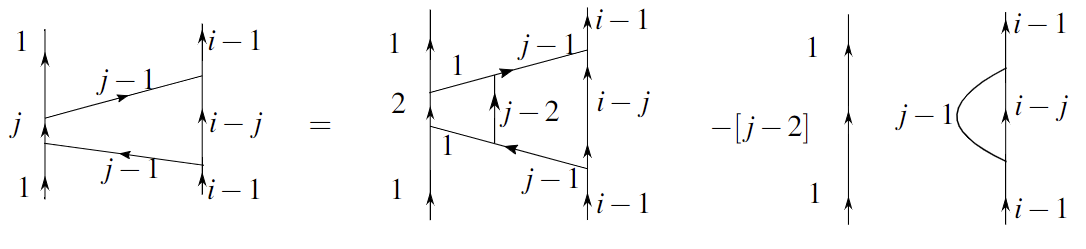}\\
\hspace*{4cm}\includegraphics[scale=.6]{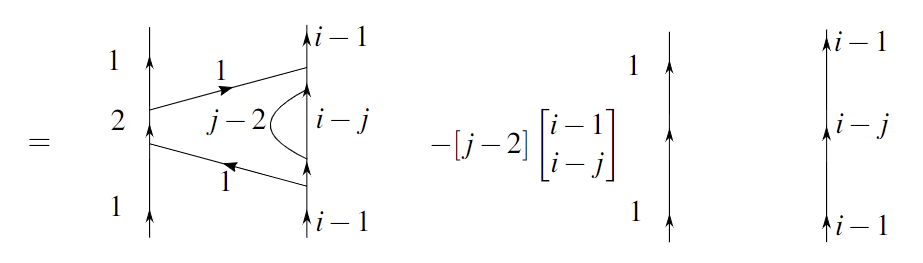}

Now going back to (\ref{crazybubble1}), apply the relation corresponding to (\ref{i.2})  again for the labels `$i-1$' appearing on the two (right) middle edges of the first diagram on RHS. Under applying $6j$ moves and the box relation (\ref{boxrelationsckm}), we get the following: \\

\def\svgwidth{14cm}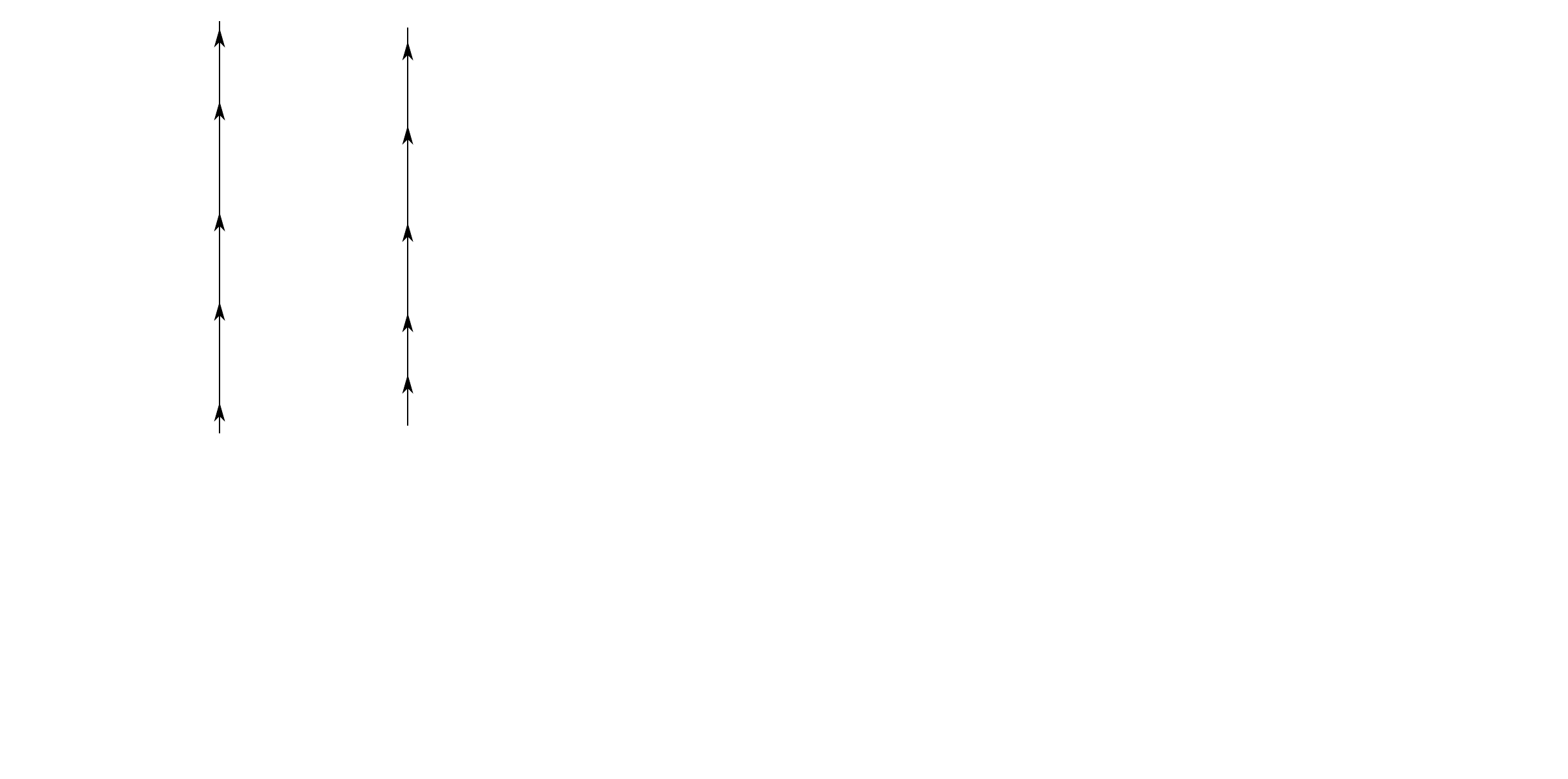\\

where, $A'=A^2 [2]\begin{bmatrix} i-2 \\ i-j\end{bmatrix}$, \text{ }$A'' = A^2 \left(\begin{bmatrix} i-2 \\ i-3\end{bmatrix} -[j-2] \begin{bmatrix} i-1 \\ i-j\end{bmatrix}-[2][i-3]\right)[3-i]$ and $B' = A'[3-i]+A''$.
\\

Now, apply the relation corresponding to (\ref{i.2}) on the labels `$3$' on the left edges of the boxes, starting with the lower one first. This followed by applying the box relation \ref{boxrelationsckm} on the newly obtained interior box finally gives us the desired outcome once we note the following:







\def\svgwidth{14.57cm}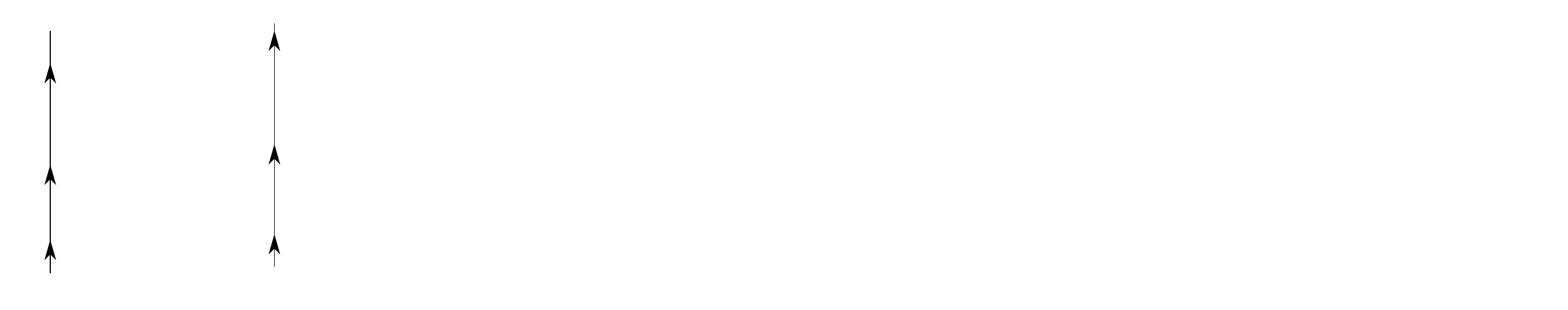\\

Notice that we picked the most complicated diagram out of the four diagrams in the original expansion, but the other three diagrams on RHS are very similar (and simpler) to the diagram we worked out. Thus, the steps that led us to simplify the first diagram works for the rest of the diagrams as well. Finally, the result follows by combining the coefficients from the like terms. 

\textit{The bubble relation II:} Use relation (\ref{boxrelationsckm}) again to get:
\begin{center}
    \includegraphics[scale=.6]{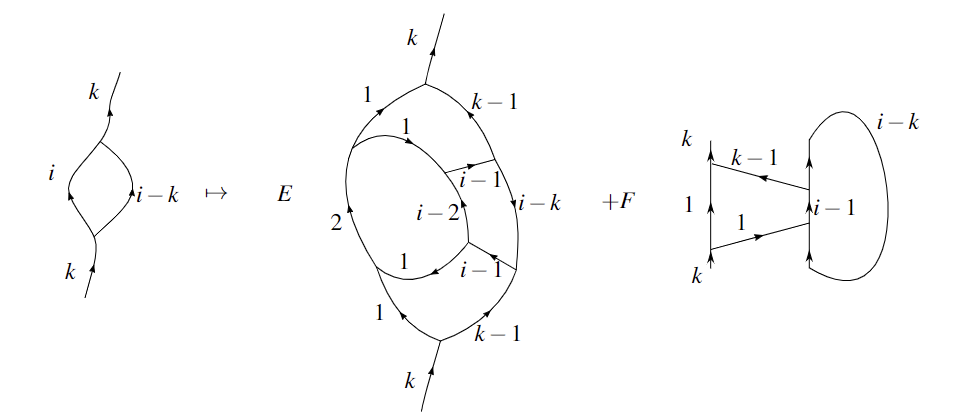}\\
\end{center}
where $E= \dfrac{1}{[k]^2}$ and $F=-\dfrac{[i-2]}{[k]^2}$.
Finally, the result is obtained by applying the box relation below.\\
\includegraphics[scale=.7]{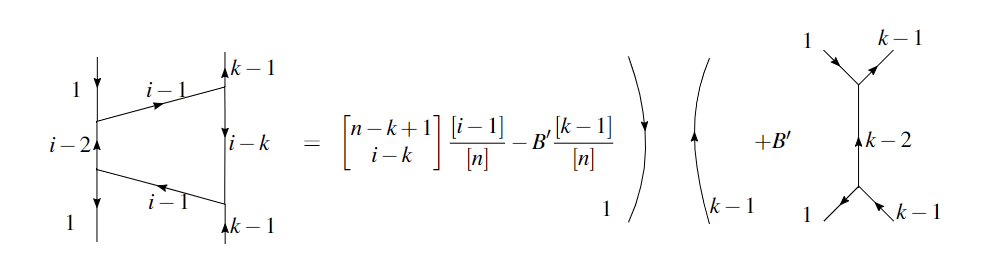}
\\
where $B' = \begin{bmatrix} n-k \\ i-k\end{bmatrix}$.
Note that the above coefficients can be computed using locality and the following relation:
\begin{center}
    \includegraphics[scale=.65]{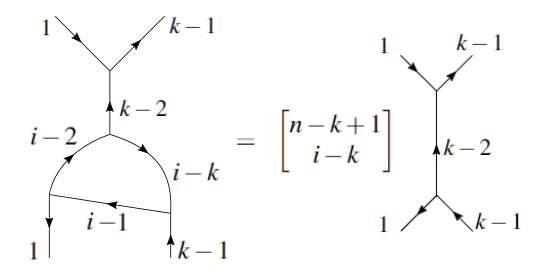}
\end{center}

Finally, surjectivity of $\phi_i$ is obtained by recalling the definition of $\phi_i$ and noting that $\mathcal{MOY}^{i-1}(SL_n) \subseteq \mathcal{MOY}^i(SL_n)$. 
\end{proof}

\begin{theorem}\label{reducecolorthm} 
For $n\geq i\geq 2$, $\mathcal{MOY}^i(SL_n)$ and $\mathcal{MOY}^{i-1}(SL_n)$ are equivalent as spherical braided (ribbon) categories.
\end{theorem}
\begin{proof}
For the case of $i=3$, the $6j$ relation in lower labels (`$1$' and `$2$') is obtained as in the proof of Lemma \ref{faithfulfunctorphi} which is equivalent to the relation shown below. Note that this relation is also a consequence of the Reidemeister III invariance (after resolving the crossings) with labels $`1$'.
\begin{center}
    \def\svgwidth{12cm}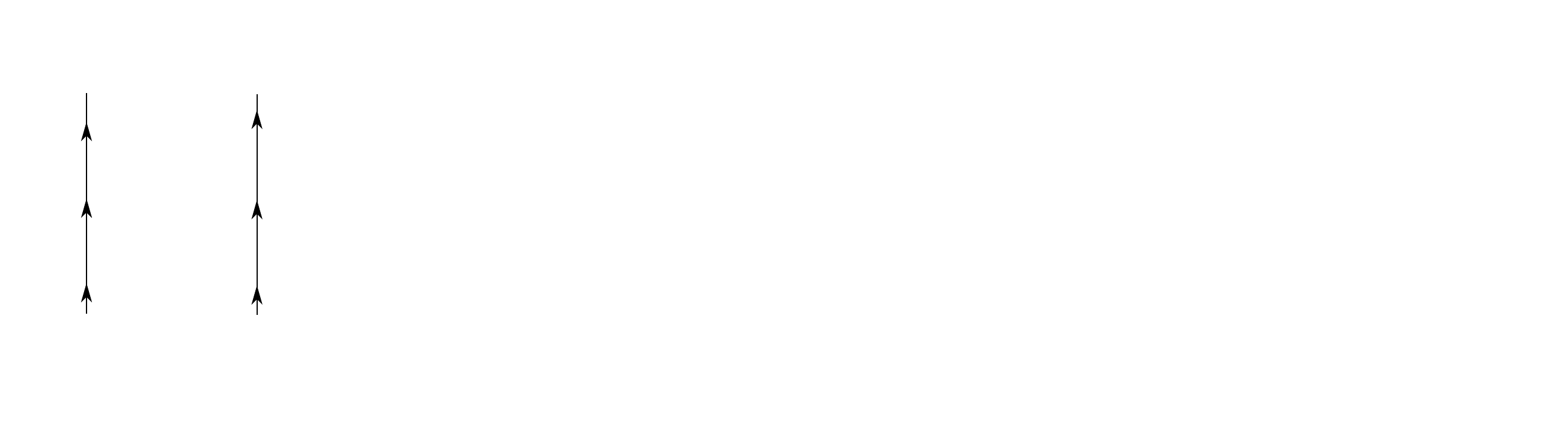
\end{center}

Similarly, the bubble relations in lower labels ($`1$' and $`2$') are obtained again as in the proof of Lemma \ref{faithfulfunctorphi} where one starts by replacing any graph containing an edge with label `$3$' in terms of lower labels (as demonstrated in L.H.S. above) and obtain an equivalent relation in terms of lower labels.

The case of $i=2$ also follows accordingly by noting that any edge with label `2' in $\mathcal{MOY}^{2}(SL_n)$ can be written as a linear combination of diagrams (involving a crossing) with labels `1'. 

\begin{align}
 \includegraphics[scale=.6]{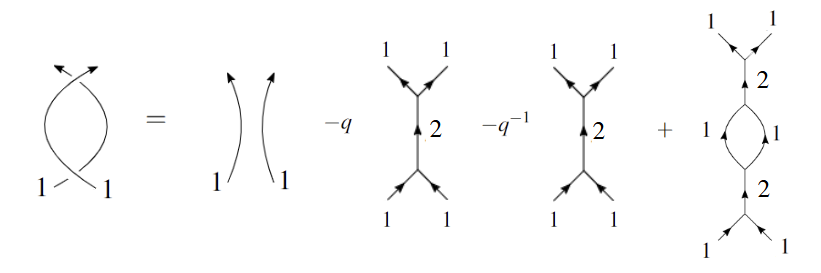}   
\end{align}
Note that in $\mathcal{MOY}^{2}(SL_n)$, each edge with label `$2$' appears in between label `$1$'s. This then gives us that the two Reidemeister II relations (involving labels `$1$'), one where the two strands are oriented the same way and the other where they are oriented the opposite way, imply the two bubble relations (\ref{ckmb1}, \ref{ckmb2}) respectively. An instance of this is shown above.
Thus, the result now follows from Lemma \ref{faithfulfunctorphi}.
\end{proof}



\section{From the Sikora Category to CKM}\label{siktockmmain}
\noindent Let $\mathcal{S}p^n(SL_n)$ be the full subcategory of the CKM spider category with objects \{$1^\pm$\}. In this section, we define a braided spherical functor: $\tau:\Tilde{\mathcal{S}} \to \mathcal{S}p^n(SL_n)$. Our goal is to show that this functor is fully faithful.
Note that $\mathcal{S}p^n(SL_n)$ contains all the tag relations as the trivalent vertex given by the projection on (and inclusion of) the trivial representation $\Lambda^n \CC^n$ is now one of the generators.  

Also throughout this section, whenever $q$ is a root of unity, it is assumed that the quantum integers $[1],\ldots,[n]$ are invertible.
\newpage
\subsection{On generators} We define $\tau$ as follows on the generators of morphisms:
\begin{center}\label{tauongens}
 \includegraphics[scale=.75]{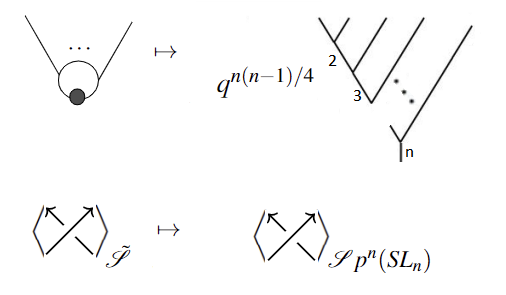}
\end{center}
Note, the assignment of the $n$-vertex from $\Tilde{\mathcal{S}}$ to the chosen (fusion) tree diagram above is unique up to $6j$ moves. Further, the tree diagram on the image is normalized so that it is $q^{n(n-1)/4}$ times the respective tree diagram in $\mathcal{S}p^n(SL_n)$.
Also note that the label `$n$' here shows the placement of the tag. 

\subsection{On relations}
We start by recalling a result from \cite{bla}. Below, $H_j$ represents the Hecke algebra, i.e. the algebra isomorphic to the quotient of the braid group algebra $k[B_j]$ by the HOMFLYPT relation.

Let $\sigma_i \in H_n, i= 1,\cdots, n-1$ represent the standard generators of the braid group $B_n$ where the strand labeled $i$  crosses over the one labeled $i+1$. Further, let $f_j$ and $g_j$ be the deformation of Young symmetrizer and anti-symmetrizer, respectively, where $f_j$ satisfies the recursive relations (\ref{freln1}) and (\ref{freln2}), and $g_j$ satisfies the relation (\ref{antisym}) below. 

\begin{align}
    [2]f_2 &=s^{-1}\OIdma + a^{-1}\OPxa \label{freln1}\\
    [n+1]f_{n+1}&=-[n-1]f_n\otimes 1_1 + [2][n](f_n\otimes 1_1)(1_{n-1}\otimes f_2)(f_n \otimes 1_1) \label{freln2}\\
     g_j &= \frac{1}{[j]!}s^{j(j-1)/2}\sum_{\pi \in S_j}(-as)^{-l(\pi)}w_\pi\label{antisym}
\end{align}
where $l(\pi)$ is the length of $\pi$ and $w(\pi)$ is the positive braid associated with the permutation $\pi$. \\
Note that, just as $f_j$, $g_j$ can also be uniquely defined using the following recursive formulas \cite{bla}:
\begin{align}
    g_1 &= 1_1 \label{antisymrecur1}\\
g_{j+1} &= 1_1 \otimes g_j - \frac{[2][j]}{[j+1]}(1_1 \otimes g_j)(f_2\otimes 1_{j-1})(1_1 \otimes g_j)\label{antisymrecur2}
\end{align}

\begin{proposition}[\cite{bla}]
If $[j]!$ is invertible in $k$, then there exists a unique idempotent $f_j\in H_j$ such that $\forall_i \sigma_if_j = a s f_j=f_j\sigma_i$, and a unique idempotent $g_j\in H_j$ such that $\forall_i \sigma_i g_j = -as^{-1}g_j=g_j\sigma_i$.
\end{proposition}

\begin{remark}
Setting $a=q^{-1/n}$ and $s=q$ on equation (\ref{antisym}) above, one sees that this resembles the relation (\ref{sourcesink}) where the source-sink diagram on the L.H.S. is equal to $[n]!q^{n(n-1/2)}\cdot g_n$.
\end{remark} 

\begin{lemma}\label{antisymlemma2}
For each $1<j\leq n$, there exists an element in $\mathcal{S}p^n(SL_n)$ that satisfies the relation (\ref{antisym}) and is given by
\begin{align}\label{antisymlemmapic}
    g_j = \frac{1}{[j]!} 
   \raisebox{-19mm}{\includegraphics[width=0.175\textwidth]{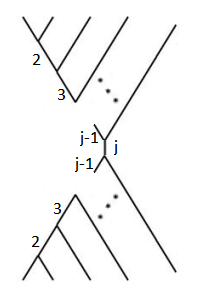}}
\end{align}
\end{lemma}
\begin{proof}
The proof will be done by showing that such an element satisfies the equivalent statements in equations (\ref{antisymrecur1}) and (\ref{antisymrecur2}) and by induction on $j$.

\noindent\textit{Base case ($j=1$):} Setting $j=1$ in equation (\ref{antisym}), we get that such an element has to be an identity on a single strand. Hence, the condition (\ref{antisymrecur1}) is satisfied. Next, to verify the condition (\ref{antisymrecur2}), use the two equations satisfied by $f_j$ and $g_j$ below, where the second equation follows from the relation (\ref{freln1}) by substituting $a=q^{-1/n}$ and $s=q$.
\begin{align}
    1_2 &= f_2 + g_2\\
    [2]f_2 &= q^{-1}\OIdma + q^{1/n}\OPxa
\end{align}
Using the two equations above to solve for $g_2$ after resolving the crossing proves the base case. 

\noindent\textit{Induction step:} Assume the claim holds for all $g_j$ for $1<j\leq m$. We will show it's true for $j=m+1$. 
Start with
\[g_{m+1} = 1_1 \otimes g_m - \frac{[2][m]}{[m+1]}(1_1 \otimes g_m)(f_2\otimes 1_{j-1})(1_1 \otimes g_m)
\]
Now use the induction hypothesis on the R.H.S.. Note that we have three diagrams stacked together in the term that is being subtracted on the RHS. Simplify this diagram by resolving the crossing obtained by substituting $f_2$ and applying the relation shown below.  

\begin{align*}
    \def\svgwidth{12.85cm}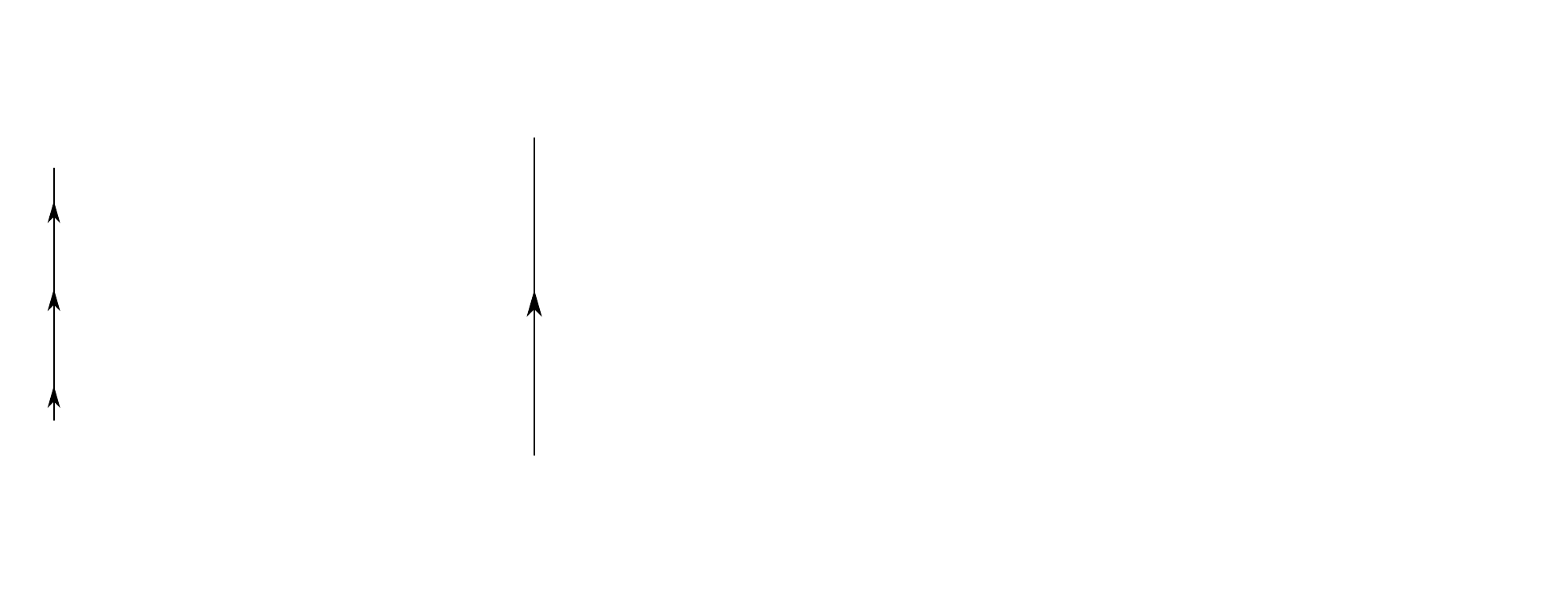
\end{align*}
Finally, combining the terms proves the claim. 
\end{proof}

\begin{lemma}\label{tauwelldef}
The functor $\tau$ is well-defined.
\end{lemma}
\begin{proof}
By resolving the crossings, one can see that the relations (\ref{homflysik}, \ref{twistsik}, \ref{unknotsik}) get mapped to relations in $\mathcal{S}p^n(SL_n)$. Further, Lemma (\ref{antisymlemma2}) tells us that there exist elements in $\mathcal{S}p^n(SL_n)$ that satisfy the RHS of the relation (\ref{sourcesink}), after some normalization. Hence, with the assignment shown below, one sees that $\tau$ respects the relation (\ref{sourcesink}). 

\begin{center}
    \includegraphics[scale=.7]{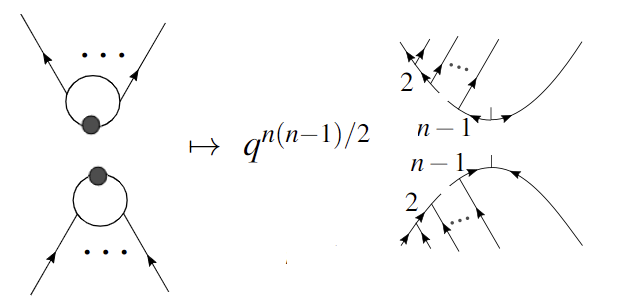}
\end{center}
Thus, this shows that $\tau$ is well-defined.
\end{proof}

\begin{lemma}\label{tauonto}
The functor $\tau$ is surjective.
\end{lemma}

\begin{proof}
For any diagram in $\mathcal{S}p^n(SL_n)$, let us consider the two cases: subgraphs with tag and without tag.
For the latter case, note that any untagged diagram in $\mathcal{S}p^n(SL_n)$ lives in $\mathcal{MOY}^n(SL_n)$ which is equivalent to $ \mathcal{MOY}^2(SL_n)$ from Lemma \ref{faithfulfunctorphi}. Further, from Lemma \ref{recursion1}, any such diagram can be written in terms of linear combination of diagrams with colors $\{1^\pm, 2^\pm\}$. Finally, any diagram with the color `$2$' can be replaced with a linear combination of diagrams with a crossing and parallel strands with color `$1$'. Recall that the preimage of a crossing with colors $\{1^\pm\}$ in $\mathcal{S}p^n(SL_n)$ is the same crossing in $\Tilde{\mathcal{S}}$.
Thus, any untagged diagram in $\mathcal{S}p^n(SL_n)$ can be obtained as an image of a combination of diagrams in $\Tilde{\mathcal{S}}$ under $\tau$.

Now, consider the case of tagged diagrams. Around each tag, we form a disk and proceed by creating bubbles on each of the two edges that meet at the tag. Using the definition of $\tau$, we obtain the preimage of a tagged diagram to be a diagram with a source or a sink. The procedure is demonstrated below.

\begin{align}\label{tagtosink}
\includegraphics[scale=.6]{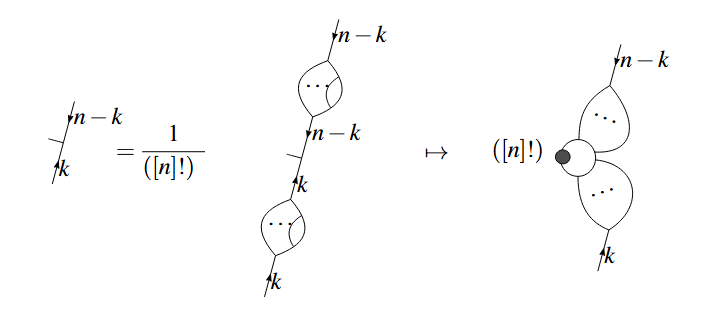}
\end{align}

For any vertices and edges that are not adjacent to the tag, we can identify those subgraphs (after applying the bubbling procedure as needed) as diagrams in $\mathcal{S}p^i_0(SL_n)$ whose preimage then lives in $\Tilde{\mathcal{S}}$ as discussed in the ``without tag" case above. 
Thus, this shows that the functor $\tau$ is surjective.

\end{proof}

\subsection{Faithfulness of $\tau$}
We prove faithfulness of $\tau$ by showing that there exists a well-defined surjective functor $\tau^{-1}: \mathcal{S}p^n(SL_n) \to \Tilde{\mathcal{S}}$.

Based on the observation we made in equation (\ref{tagtosink}) above, define $\tau^{-1}$ locally on the tagged vertex as follows:
\begin{center}
    \includegraphics[scale=.7]{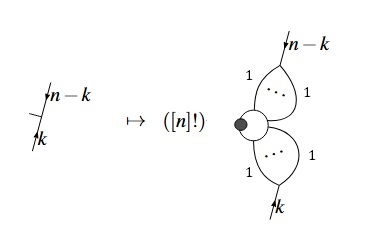}
\end{center}

Similarly, for each trivalent vertex, $\tau^{-1}$ is defined in the following manner (below each label on the RHS represents the number of parallel strands):\\
\begin{align}\label{trivalentvertextonvalent}
     \includegraphics[scale=.55]{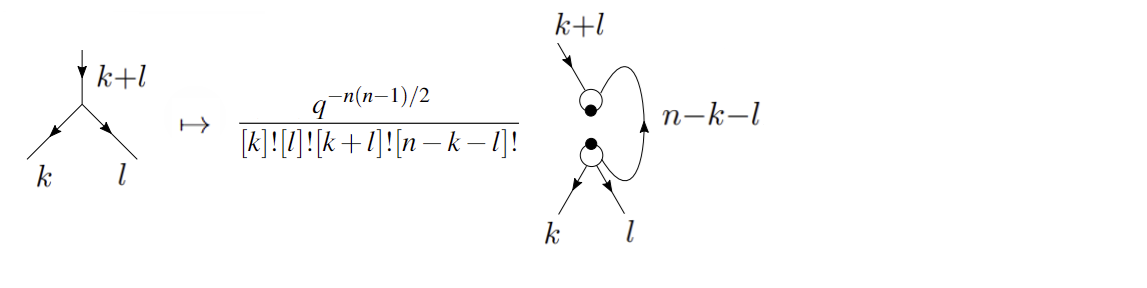}   
\end{align}

For a trivalent vertex where all the arrows in LHS of (\ref{trivalentvertextonvalent}) are reversed, $\tau^{-1}$ assigns it to an $n-$valent tangle with arrows reversed on the RHS. Note that the coefficient on the RHS stays the same as long as the corresponding labels on the vertex stay the same.

For crossings involving labels `$1$', $\tau^{-1}$ takes crossings in $\mathcal{S}p^n(SL_n)$ to those in $ \Tilde{\mathcal{S}}$. For closed knot components with label `$k\geq 2$', one first applies the bubbling procedure to introduce a pair of trivalent vertices. Then use (\ref{trivalentvertextonvalent}) to map it into $ \Tilde{\mathcal{S}}$.    

\begin{remark}
    Here we provide some explanation to motivate the assignment of $\tau^{-1}$ on a trivalent vertex as defined above. 
    To begin with, using the antisymmetrizer relation in (\ref{antisymlemmapic}), we had the following assignment of $\tau$. 
\begin{align}\label{antisymfortag1}
    \includegraphics[scale=.72]{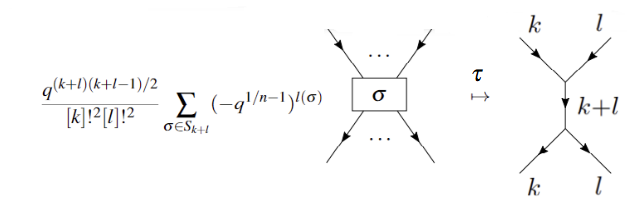}
\end{align}
We start with the trivalent vertex on the LHS of (\ref{trivalentvertextonvalent}), and apply the bubbling procedure on each edge as observed in (\ref{tagtosink}). Then, we use the following fact from Lemma $25$ in \cite{sik}.
\begin{align}\label{sikoralemma1}
 \includegraphics[scale=.7]{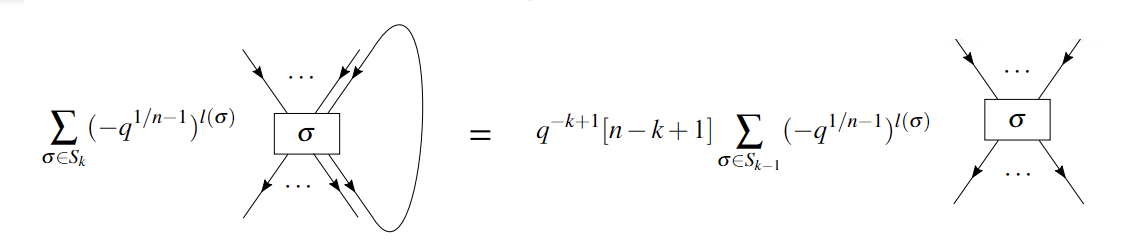}
\end{align}
Use induction on the number of strands being traced to the right on relation (\ref{sikoralemma1}) above, and make substitution using the relation (\ref{sourcesink}) to obtain the following equation (below, the labels on LHS represent the number of parallel strands):
\begin{align}\label{sikoralemma2}
\includegraphics[scale=.7]{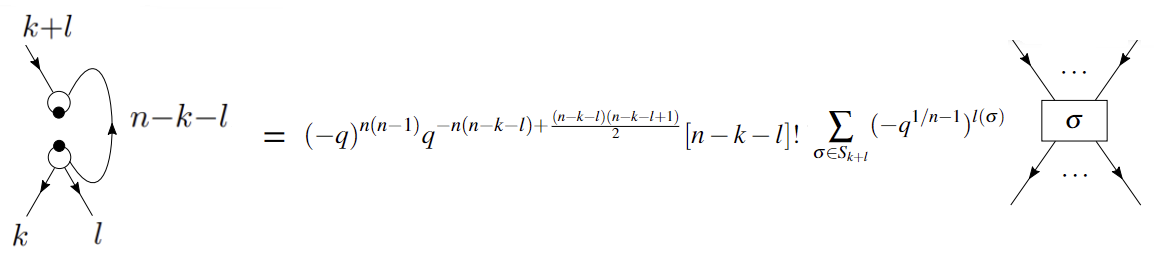}
\end{align}
Finally, simplify equation (\ref{sikoralemma2}) and compare with (\ref{antisymfortag1}). This shows that with this definition, $\tau^{-1}$ does indeed behave as the inverse of $\tau$.

\end{remark}

\begin{lemma}\label{taufaithful}
$\tau^{-1}$ is well-defined.
\end{lemma}
\begin{proof}
Let us again consider the two cases: relations with and without tags. \\
For the latter case, since all the non-tagged relations are contained in $\mathcal{R}el^i(SL_n)$,
Theorem \ref{reducecolorthm} tells us that it suffices to consider the relations in $\mathcal{MOY}^n(SL_n) \cong \mathcal{MOY}^2(SL_n)$. Then the arguments presented in Theorem \ref{reducecolorthm} tell us how the bubble and $6j$ relations in $\mathcal{MOY}^2(SL_n)$ follow from the Reidemeister II and III relations respectively. Further, by writing the edge colored with `2' as the combination of diagrams with labels `1' (involving crossings), one shows that the braid relations in $\mathcal{MOY}^2(SL_n)$ also follows from the Reidemeister invariance and the specialized HOMFLYPT relation involving the label `1'. Hence, from Theorem \ref{allckmrelationtheorem1}, the image of all untagged relations under $\tau^{-1}$ lie in the set of relations in $\Tilde{\mathcal{S}}$.


In the case of relations with tags, there are three types of relations to be considered. Consider one of the two `tag migration' relations: the one that is a $6j$ relation (c.f. (\ref{6jckm1})) involving the color `$n$'. Recall the definition of $\tau$ on the generators. In particular, $\tau$ maps an $n-$web that is a sink (similarly, source) to a left-adjoint tree with a tag that is a sink (accordingly, a source), c.f. (\ref{tauongens}). Note that the choice for the image of the $n-$web is unique upto $6j$ move involving the color `$n$'. Hence, by construction, $\tau^{-1}$ sends all diagrams that are related by $6j$ moves involving the color `$n$' to the corresponding $n-$web which is unique in $\Tilde{\mathcal{S}}$.

Now consider the `tag switch' relation (\ref{tagswitchckm}). In order to understand how this relation is obtained in $\Tilde{\mathcal{S}}$, first observe the following:

\begin{align}\label{tagswitchsik}
\def\svgwidth{7cm}
\begingroup%
  \makeatletter%
  \providecommand\color[2][]{%
    \errmessage{(Inkscape) Color is used for the text in Inkscape, but the package 'color.sty' is not loaded}%
    \renewcommand\color[2][]{}%
  }%
  \providecommand\transparent[1]{%
    \errmessage{(Inkscape) Transparency is used (non-zero) for the text in Inkscape, but the package 'transparent.sty' is not loaded}%
    \renewcommand\transparent[1]{}%
  }%
  \providecommand\rotatebox[2]{#2}%
  \newcommand*\fsize{\dimexpr\f@size pt\relax}%
  \newcommand*\lineheight[1]{\fontsize{\fsize}{#1\fsize}\selectfont}%
  \ifx\svgwidth\undefined%
    \setlength{\unitlength}{722.37875631bp}%
    \ifx\svgscale\undefined%
      \relax%
    \else%
      \setlength{\unitlength}{\unitlength * \real{\svgscale}}%
    \fi%
  \else%
    \setlength{\unitlength}{\svgwidth}%
  \fi%
  \global\let\svgwidth\undefined%
  \global\let\svgscale\undefined%
  \makeatother%
  \begin{picture}(1,0.23215839)%
    \lineheight{1}%
    \setlength\tabcolsep{0pt}%
    \put(0,0){\includegraphics[width=\unitlength,page=1]{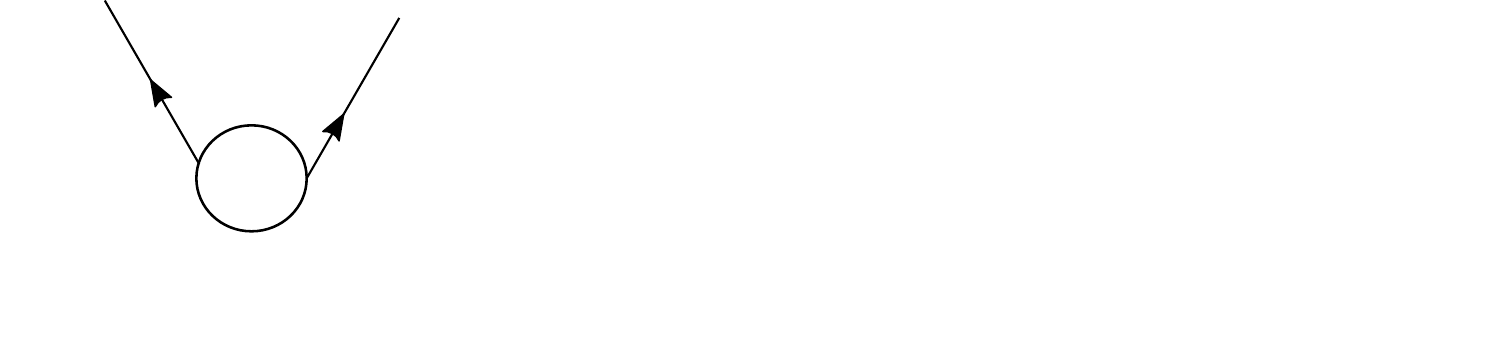}}%
    \put(0.14696517,0.16101218){\color[rgb]{0,0,0}\makebox(0,0)[lt]{\lineheight{1.25}\smash{\begin{tabular}[t]{l}$\cdots$\end{tabular}}}}%
    \put(0,0){\includegraphics[width=\unitlength,page=2]{tagswitchsik.pdf}}%
    \put(0.82049921,0.15285106){\color[rgb]{0,0,0}\makebox(0,0)[lt]{\lineheight{1.25}\smash{\begin{tabular}[t]{l}$\cdots$\end{tabular}}}}%
    \put(0,0){\includegraphics[width=\unitlength,page=3]{tagswitchsik.pdf}}%
    \put(0.40837578,0.09983205){\color[rgb]{0,0,0}\makebox(0,0)[lt]{\lineheight{1.25}\smash{\begin{tabular}[t]{l}$=(-1)^{n-1}$\end{tabular}}}}%
    \put(0,0){\includegraphics[width=\unitlength,page=4]{tagswitchsik.pdf}}%
  \end{picture}%
\endgroup%

\end{align}
The relation above follows from the relations (\ref{homflysik} - \ref{unknotsik}) in $\Tilde{\mathcal{S}}$. We refer the reader to \cite{sik} for more details. This relation tells us how the source (similarly, sink) moves past a strand labeled `$1$'. Using our definition of $\tau$ on generators, this tells us that $\tau^{-1}$ maps the relation (\ref{tagswitchckm}) with $k=1$ to the relation (\ref{tagswitchsik}) above. In order to obtain the general tag switch relation from this, one proceeds as in (\ref{tagtosink}) shown above by creating bubbles around the tag and migrating the tag to the strands with label `$1$', then repeatedly applying the relation (\ref{tagswitchsik}). It's worth noting that while doing this procedure, on the initial step, there is a choice to be made regarding which one of the two edges (labelled $k$ or $n-k$) to migrate the tag on. However, this doesn't make a difference since $(-1)^{k(n-k)}$ is always positive for $n$ odd and for $n$ even, it's enough to know the parity of either $k$ or $n-k$ as both $(-1)^k$ and $(-1)^{n-k}$ yield the same value. Thus the relations (\ref{homflysik} - \ref{unknotsik}) in $\Tilde{\mathcal{S}}$ imply the tag switch relation (\ref{tagswitchckm}). 

Finally, consider the ``tag migration'' relation (\ref{movetagckm}). Make the following observation using the definition of $\tau^{-1}$ (c.f. (\ref{trivalentvertextonvalent})).

\begin{align}
\raisebox{-15mm}{\includegraphics[width=0.12\textwidth]{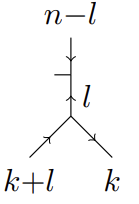}} 
\longmapsto \dfrac{q^{-3n(n-1)/4}}{[k]![l]![k+l]![n-l]![n-k-l]!}
\raisebox{-19mm}{\includegraphics[width=0.27\textwidth]{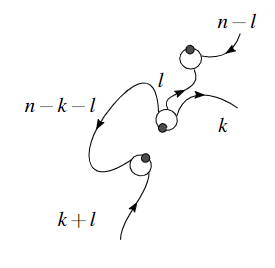}}
\label{lhstagrel}\\
\raisebox{-15mm}{\includegraphics[width=0.175\textwidth]{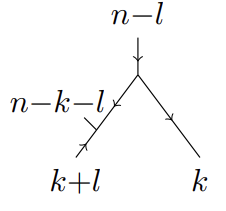}} 
\longmapsto \dfrac{q^{-3n(n-1)/4}}{[k]![l]![k+l]![n-l]![n-k-l]!}
\raisebox{-21mm}{\includegraphics[width=0.23\textwidth]{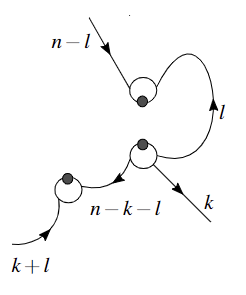}} 
\label{rhstagrel}
\end{align}

Then, notice that starting from the diagram on the image of $\tau^{-1}$ in (\ref{lhstagrel}), one obtains the diagram on (\ref{rhstagrel}) by applying the tag-switch relation twice on the two sinks connecting the edge labeled $n-k-l$. As the tag-switch relation on the same number of strands is applied twice, no negative sign appears. Hence, this shows that the tag migration relation (\ref{movetagckm}) is implied by the relations in $\Tilde{\mathcal{S}}$. Thus, this shows $\tau^{-1}$ is well-defined.
\end{proof}

\begin{theorem}\label{sikorackmequiv}
$\Tilde{\mathcal{S}}$ and $\mathcal{S}p^n(SL_n)$ are equivalent as ribbon tensor categories.
\end{theorem}
\begin{proof}
This follows immediately from Lemmas \ref{tauwelldef}, \ref{tauonto} and \ref{taufaithful}.
\end{proof}

\begin{corollary}\label{sikorapresentation1}
The categories $\Tilde{\mathcal{S}}$ and $\mathcal{C}_n$ are equivalent as ribbon tensor categories.
\end{corollary}
\begin{proof}
First, define a functor $\Gamma^n: \mathcal{S}p^n(SL_n) \to \mathcal{C}_n$, where $\Gamma^n$ is a restriction of the functor to $\mathcal{S}p^n(SL_n)\subseteq \mathcal{S}p(SL_n)$ constructed in \cite{ckm} that goes from $\mathcal{S}p(SL_n) \to \mathcal{R}ep(SL_n)$.
Note that the main result from \cite{ckm} immediately implies that as a spherical tensor functor, this gives an equivalence of the two categories. In fact, fullness of $\Gamma^n$ as a braided tensor functor follows from the proof of fullness of their main result in \cite{ckm} along with the definition of braiding in Section 6 in \cite{ckm}. In order to check faithfulness of $\Gamma^n$ as a braided tensor functor, notice that from the final corollary in Section 6 of \cite{ckm}, it is known that the braiding can be expressed as a linear combination of boxes. Further, from Theorem \ref{allckmrelationtheorem1} it's known that the box relations are equivalent to the Reidemeister relations (check also for example \cite{moy}). Hence, this tells us that $\Gamma^n$ is faithful. Thus, $\Gamma^n$ gives an equivalence of $\mathcal{S}p^n(SL_n)$ with $\mathcal{C}_n$ as ribbon tensor categories.

Now, define a functor $RT^n: \Tilde{\mathcal{S}} \to \mathcal{C}_n $ where, $RT^n:= \Gamma^n \circ \tau$. Then the fact that $\Gamma^n$ is an equivalence of ribbon categories together with Theorem \ref{sikorackmequiv} implies that $RT^n$ is an equivalence of ribbon tensor categories. 

\end{proof}

\begin{theorem}[Proof of Conjecture \ref{lesikconjecture}]\label{proofofconjecture}
\end{theorem}
\begin{proof}
The proof can be understood using the following diagram whose details we provide below.

\begin{center}
\begin{tikzcd}
\mathcal{S}p^n(SL_n) \arrow[r, "\Gamma^n"] & \mathcal{C}_n \\
 \Tilde{\mathcal{S}} \arrow[u, "\cong"] \arrow[ru, "RT^n"]\\
\mathfrak{S}_n^b \arrow[u, "\tilde{\pi}"] \arrow[uur, bend right, "RT_0"]
\end{tikzcd}
\end{center}

Let $\tilde{\pi}$ be a functor from $\mathfrak{S}_n^b \to \Tilde{\mathcal{S}}$ given by the quotient of $\mathfrak{S}_n^b$ by the relations (\ref{homflysik}--\ref{unknotsik}). Note that
from the Corollary \ref{sikorapresentation1}, we have that the category $\mathfrak{S}_n^b/\ker(\tilde{\pi}) = \Tilde{\mathcal{S}}$ provides a presentation in terms of generators and relations of $\mathcal{C}_n$. Also recall that the full subcategory of $\mathcal{R}ep(SL_n)$ where the objects are finite tensor products of the standard representations is unique. Further, from \cite{lesik}, we know that the functor $RT_0$ is surjective. Thus, using the  Corollary \ref{sikorapresentation1} we get the following categorical equivalences
\[\mathfrak{S}_n^b/\ker(RT_0) \cong \mathfrak{S}_n^b/\ker(\tilde{\pi}) \cong \mathcal{C}_n
\]

It follows that $\ker(RT_0)$ is the monoidal ideal generated by elements given in relations (\ref{homflysik}--\ref{unknotsik}).


\end{proof}

\section{Skein module isomorphism}\label{skeinmodiso}
In this section, we recall the definition of a\textit{ skein category} coming from a ribbon tensor category.
We then make a comparison between the skein categories associated to the ribbon categories $\Tilde{\mathcal{S}}$, $\mathcal{S}p^n(SL_n)$ and $\mathcal{S}p(SL_n)$. We conclude the section by showing that the $SL_n$-skein modules associated to these skein categories are isomorphic. 

Below we recall the definition of a skein category \cite{jf, cooke, gjs}. 

\begin{definition}
    Let $\mathcal{A}$ be a ribbon tensor category over $R$ and $\Sigma$ an oriented surface. The \textit{skein category} $\mathbf{Sk}_\Sigma^\mathcal{A}$ is defined using the following data:
    \begin{itemize}
        \item A finite collection of oriented embedding of disks, $x_1,\ldots, x_n: \mathbb{D}\to \Sigma$, each labeled by objects in $\mathcal{A}$. The \textit{objects} in $\mathbf{Sk}_\Sigma^\mathcal{A}$ are then given by the choice of $x-$axis, $\vec{x_i}$, in each of these oriented labeled disks. 
        \item Let $\vec{X}:= \bigcup_i\Vec{x_i}$ and $\vec{Y}:= \bigcup_i\Vec{y_i}$ be objects in $\mathbf{Sk}_\Sigma^\mathcal{A}$. The \textit{morphisms} from $\vec{X}$ to $\vec{Y}$ are given by $R-$span of isotopy classes (relative to the boundary) of labeled ribbon graphs that represent morphisms in $\mathcal{A}$ \cite{rt} living in $\Sigma \times [0,1]$ such that the graph intersects the surface $\Sigma \times \{0\}$ at $\vec{X}$ and similarly intersects $\Sigma \times \{1\}$ at $\vec{Y}$. The ribbon graphs come with coupons which are embedded rectangles $I \times I$ (where $I$ is an interval) that represent a morphism in $\mathcal{A}$ from the ordered tensor product of incoming edges to the ordered tensor product of outgoing edges. Further, the morphisms satisfy local skein relations coming from the ribbon category $\mathcal{A}$ in a ball $\mathbb{D} \times I$ in the interior of $\Sigma \times [0,1]$. 
        \end{itemize}
    Note that the composition of morphisms in $\mathbf{Sk}_\Sigma^\mathcal{A}$ is given by stacking the embedded ribbon graphs on top of another and retracting to $\Sigma \times [0,1]$.
\end{definition}

Consider two equivalent ribbon tensor categories $\mathcal{A}_1$ and $\mathcal{A}_2$ over $R$, where the equivalence is given by a fully faithful ribbon tensor functor $\Tilde{\tau}:\mathcal{A}_1 \to \mathcal{A}_2$.   
\begin{theorem}\label{skeincatequiv}
   Let $\Sigma$ be an oriented surface, then the skein categories $\mathbf{Sk}_\Sigma^{\mathcal{A}_1}$ and $\mathbf{Sk}_\Sigma^{\mathcal{A}_2}$ are equivalent as $R-$linear categories.
\end{theorem}
\begin{proof}
We will prove this by showing that there exist two well-defined functors $\Tilde{\tau}_{\Sigma}: \mathbf{Sk}_\Sigma^{\mathcal{A}_1} \to \mathbf{Sk}_\Sigma^{\mathcal{A}_2}$ and $\Tilde{\tau}_{\Sigma}^{-1}: \mathbf{Sk}_\Sigma^{\mathcal{A}_2} \to \mathbf{Sk}_\Sigma^{\mathcal{A}_1}$ between these quotient categories. The $R$-linear functor $\Tilde{\tau}_{\Sigma}$ sends each object, recall these are labels on an oriented disk, $\Vec{x_i}$, to $\Tilde{\tau}(\Vec{x_i})$. This functor is essentially surjective as the functor $\Tilde{\tau}$ is essentially surjective.


Recall that any two morphisms $f,g \in \Hom_{\mathbf{Sk}_\Sigma^{\mathcal{A}_i}}(\Vec{X}, \Vec{Y})$ are same if either $f$ is isotopic relative to the boundary to $g$ or they are related to one another in a $3$-ball by a local relation in $\mathcal{A}_i$ and identical elsewhere. 
Since the relations are local, it suffices to analyze the fullness and faithfulness of our functor for relations 
between framed labelled graphs locally on a 3-ball.
Assume there exists a relation of morphisms in $\Hom_{\mathbf{Sk}_\Sigma^{\mathcal{A}_i}}(\Vec{X}, \Vec{Y})$ such that $\sum_\alpha f_\alpha = 0$. Then by our definition there exists a choice of $3$-balls such that $\sum_\alpha (f_\alpha \cap (\mathbb{D} \times I)) = 0$. Note that the construction in \cite{Tur2} gives us a well-defined bijection between the oriented labeled ribbon graphs in each $\mathbb{D}\times I$ and morphisms in $\Hom_{\mathcal{A}_i}(\bigotimes_{k=1}^n V_i,\bigotimes_{j=1}^m V_j)$. 
Thus, due to well-definedness of $\Tilde{\tau}$, and the fact that the relations are local, we get the well-definedness of $\Tilde{\tau}_{\Sigma}$. Similarly, the fullness of $\Tilde{\tau}_{\Sigma}$ follows due to that of $\Tilde{\tau}$.
Analogously, using the existence of the well-defined and full functor $\Tilde{\tau}^{-1}: \mathcal{A}_2 \to \mathcal{A}_1$ (by our assumption), one can show that $\Tilde{\tau}_{\Sigma}^{-1}$ is also well-defined and full.

Since all the relations are local and take place inside some $\mathbb{D}\times I$, this gives us a fully faithful $R$-linear functor $\Tilde{\tau}$.

\end{proof}

Now consider the ribbon tensor categories $\Tilde{\mathcal{S}}$ and $\mathcal{S}p^n(SL_n)$. Recall that $\Tilde{\mathcal{S}}$ is the spider category with Sikora's web relations and $\mathcal{S}p^n(SL_n)$ is the full subcategory (tensor generated by the standard representation and its dual) of the spider category with the CKM web relations. These were shown to be equivalent in Theorem \ref{sikorackmequiv}. Consequently, by Theorem \ref{skeincatequiv}, 
we get the following result:
\begin{corollary}\label{ckmskeincatequiv}
    Let $\Sigma$ be an oriented surface and $R$ be an integral domain (c.f. Section \ref{prelim}) where the quantum integers $[1],\ldots, [n]$ are invertible. We get the following equivalence of $R-$linear categories
    \[
    \mathbf{Sk}_\Sigma^{\Tilde{\mathcal{S}}} \cong \mathbf{Sk}_\Sigma^{\mathcal{S}p^n(SL_n)}
    \]
\end{corollary}

By definition, a skein cateory $\mathbf{Sk}_\Sigma^\mathcal{A}$ has a unit object $\mathbf{1}\in \mathbf{Sk}_\Sigma^\mathcal{A}$ given by empty disk labelings. Now, the \textit{skein algebra} of $\Sigma$ is given by the endomorphism algebra, $\End_{\mathbf{Sk}_\Sigma^\mathcal{A}}(\mathbf{1})$. 

Corollary \ref{ckmskeincatequiv} then gives us the following result.
\begin{corollary}\label{ckmskeinalgequiv}
    The skein algebras $\mathbf{Skalg}_\Sigma^{\Tilde{\mathcal{S}}}$ and $\mathbf{Skalg}_\Sigma^{\mathcal{S}p^n(SL_n)}$ are isomorphic as algebras over $R$, which is an integral domain (c.f. Section \ref{prelim}) where the quantum integers $[1],\ldots, [n]$ are invertible.
\end{corollary}

\begin{definition}\label{skeinmoddef}
    Let $M$ be an oriented $3$-manifold. The $\mathcal{A}$-\textit{skein module} is the $R$-module spanned by isotopy classes of closed $\mathcal{A}$-colored ribbon graphs in $M$ taken modulo the skein relations determined by any oriented ball $\mathbb{D}\times I \subset M$, denoted by $\mathbf{Skmod}_\mathcal{A}(M)$. 
\end{definition}

Note that the relative $\mathcal{A}$-skein module was defined in \cite{gjs}, which is a more general notion of a skein module. However, for our purpose in this paper, we will only be working with the $\mathcal{A}$-skein module. 

\begin{corollary}\label{ckmskeinmodequiv}
The skein modules $\mathbf{Skmod}_{\Tilde{\mathcal{S}}}(M)$ and $\mathbf{Skmod}_{\mathcal{S}p^n(SL_n)}(M)$ are isomorphic as modules over $R$, which is an integral domain (c.f. Section \ref{prelim}) where the quantum integers $[1],\ldots, [n]$ are invertible.
\end{corollary}

\begin{proof}

Define a morphism $\hat{\tau}: \mathbf{Skmod}_{\Tilde{\mathcal{S}}}(M) \to \mathbf{Skmod}_{\mathcal{S}p^n(SL_n)}(M)$ where we take each closed oriented $n$-valent ribbon graph, $\gamma$, and assign to it a linear sum of oriented trivalent ribbon graphs labeled with admissible edges in $\mathcal{S}p^n(SL_n)$ using the functor $\tau$ as defined in \ref{tauongens}. Note that this assignment is done in the same manner as in the proof of Theorem \ref{skeincatequiv} above. Call this $\hat{\tau}(\gamma)$. 

In order to check well-definedness of $\hat{\tau}$, recall that each relation between the closed framed webs take place in a ball. Consequently, for each skein relation, using the fact that these are framed webs, one obtains a thickened surface, $\Sigma_{\sqcup_{\alpha} \gamma_{\alpha}}$ containing the webs 
$\{\gamma_{\alpha}\}$. Note that our recipe in the proof of Theorem \ref{skeincatequiv} along with Corollary \ref{ckmskeinalgequiv} now gives us an isomorphism of skein modules $\mathbf{Skmod}_{\Tilde{\mathcal{S}}}(\Sigma_{\sqcup_{\alpha} \gamma_{\alpha}}) \cong \mathbf{Skmod}_{\mathcal{S}p^n(SL_n)}(\Sigma_{\sqcup_{\alpha} \tau(\gamma_{\alpha})})$, where recall that $\tau: \Tilde{\mathcal{S}} \to \mathcal{S}p^n(SL_n)$ was our fully faithful functor. 
Hence, $\hat{\tau}|_{\Sigma_{\sqcup_{\alpha} \gamma_{\alpha}}}$ is an isomorphism. Since every relation in $\mathbf{Skmod}_{\Tilde{\mathcal{S}}}(M)$ and $\mathbf{Skmod}_{\mathcal{S}p^n(SL_n)}(M)$ takes place in some submanifold (thickened surface) where the restriction of $\hat{\tau}$ is an isomorphism, due to locality of the skein relations, this implies that $\hat{\tau}:\mathbf{Skmod}_{\Tilde{\mathcal{S}}}(M) \to \mathbf{Skmod}_{\mathcal{S}p^n(SL_n)}(M)$ is an isomorphism of skein modules.


   
\end{proof}

Recall that by $\mathcal{S}p(SL_n)$ we mean the spider category which has as objects subsequences of $\{1^{\pm}, \ldots, (n-1)^\pm\}$ and the morphisms satisfy the CKM web relations. Since the skein module $\mathbf{Skmod}_{\mathcal{S}p(SL_n)}(M)$ is generated by $R$-linear combination of \textit{closed} webs, we get the following stronger result.

\begin{theorem}\label{bigckmskeinmodequiv}
The skein modules $\mathbf{Skmod}_{\Tilde{\mathcal{S}}}(M)$ and $\mathbf{Skmod}_{\mathcal{S}p(SL_n)}(M)$ are isomorphic as modules over $R$, which is an integral domain (c.f. Section \ref{prelim}) where the quantum integers $[1],\ldots, [n]$ are invertible.
\end{theorem}

\begin{proof}
We prove this by first showing that the skein modules
$\mathbf{Skmod}_{\mathcal{S}p^n(SL_n)}(M)$ and $\mathbf{Skmod}_{\mathcal{S}p(SL_n)}(M)$ are isomorphic. Corollary \ref{ckmskeinmodequiv} then gives us the isomorphism of skein modules $\mathbf{Skmod}_{\Tilde{\mathcal{S}}}(M)$ and $\mathbf{Skmod}_{\mathcal{S}p(SL_n)}(M)$.

Since there is an inclusion of categories $\mathcal{S}p^n(SL_n) \xhookrightarrow{} \mathcal{S}p(SL_n)$ as a full subcategory, this gives us inclusion (as $R$-linear categories) of the skein categories $\mathbf{Sk}_\Sigma^{\mathcal{S}p^n(SL_n)} \xhookrightarrow{} \mathbf{Sk}_\Sigma^{\mathcal{S}p(SL_n)}$ as follows.
Consider the identity functor, $\id: \mathcal{S}p^n(SL_n) \subset \mathcal{S}p(SL_n) \to \mathcal{S}p(SL_n)$, which is an equivalence onto its image, and then apply Theorem \ref{skeincatequiv} to obtain equivalence between the skein categories $\mathbf{Sk}_\Sigma^{\mathcal{S}p^n(SL_n)}$ and $\mathbf{Sk}_\Sigma^{\mathcal{S}p(SL_n)}$.

Using Corollary \ref{ckmskeinalgequiv} and a similar argument as in the proof of Corollary \ref{ckmskeinmodequiv}, we then get an inclusion of skein modules $\mathbf{Skmod}_{\mathcal{S}p^n(SL_n)}(M) \xhookrightarrow{i} \mathbf{Skmod}_{\mathcal{S}p(SL_n)}(M)$. Now we will show that the map $i$ is also surjective, for which we will use the fact that we are working with closed webs. 

For any $k$-labeled closed component in $\mathbf{Skmod}_{\mathcal{S}p(SL_n)}(M)$, we can use the bubbling procedure shown in Fig \ref{preimagefinder} and isotopy to obtain a representative of the closed component that lies in the image of a closed web in $\mathbf{Skmod}_{\mathcal{S}p^n(SL_n)}(M)$. Note that since all the relations take place in $3$-balls, this procedure can be applied to any edge of a closed web in $\mathbf{Skmod}_{\mathcal{S}p(SL_n)}(M)$ as needed to see it as an image of a closed web in $\mathbf{Skmod}_{\mathcal{S}p^n(SL_n)}(M)$. This proves surjectivity of $i$ and, hence, we have the result.

\begin{figure}
    \centering
    \def\svgwidth{12.5cm}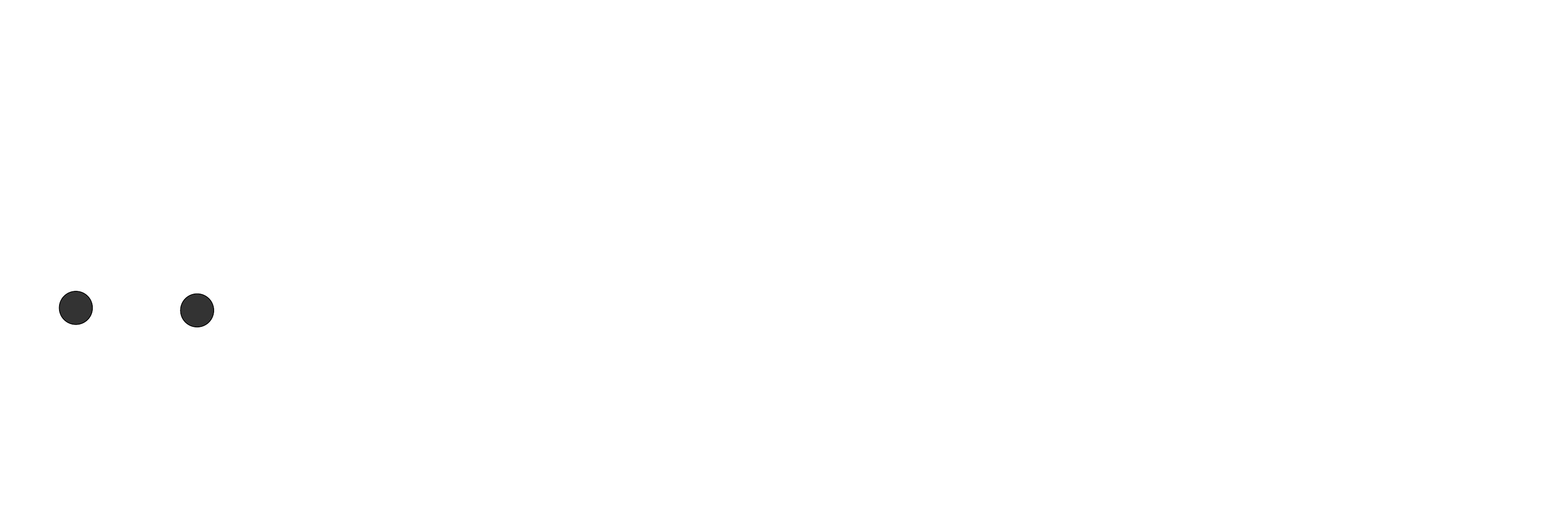
    \caption{Procedure to find a closed web in $\mathbf{Skmod}_{\mathcal{S}p^n(SL_n)}(M)$ that gives the preimage of a closed web in $\mathbf{Skmod}_{\mathcal{S}p(SL_n)}(M)$ under the map $i$.}
    \label{preimagefinder}
\end{figure}

\end{proof}

\end{document}